\documentclass{article}
\usepackage{amsmath,amsthm,amssymb,mathrsfs,amstext, titlesec,enumitem, tocloft, xcolor}
\usepackage{stmaryrd}
\usepackage{lipsum}
\makeatletter

\titleformat{\section}
{\normalfont\Large\bfseries}{\thesection}{1em}{}
\titleformat*{\subsection}{\Large\bfseries}
\titleformat*{\subsubsection}{\large\bfseries}
\titleformat*{\paragraph}{\large\bfseries}
\titleformat*{\subparagraph}{\large\bfseries}

\renewcommand{\@seccntformat}[1]{\csname the#1\endcsname.}

\renewenvironment{abstract}{%
    \if@twocolumn
      \section*{\abstractname}%
    \else 
      \begin{center}%
        {\bfseries \Large\abstractname\vspace{\z@}}
      \end{center}%
      \quotation
    \fi}
    {\if@twocolumn\else\endquotation\fi}

\makeatother
\theoremstyle{plain}
\newtheorem{thm}{Theorem}[subsection]

\theoremstyle{definition}
\newtheorem{defn}[thm]{Definition}
\newtheorem{rem}[thm]{Remark}
\newtheorem{cor}[thm]{Corollary}
\newtheorem{prop}[thm]{Proposition}
\newtheorem{example}[thm]{Example}
\newtheorem{fact}[thm]{Fact}
\providecommand{\keywords}[1]{{\bf{Keywords:}} #1}
\providecommand{\subjectclass}[2]{\textbf{Mathematics subject classification 2020:} #1}

\title{{\bf Some new symmetric structures in Ramsey theory}}

\author {Aninda Chakraborty and Sayan Goswami\footnote{Corresponding author}}

\newcommand{\Addresses}{{
  \bigskip
  \footnotesize

 Aninda Chakraborty, \textsc{Government General Degree College at Chapra, Chapra -741123, Nadia, West Bengal, India.}\par\nopagebreak
  \textit{E-mail address}: \textcolor{blue}{anindachakraborty2@gmail.com}

  \medskip

Sayan Goswami, \textsc{Department of Mathematics, 
          University of Kalyani, 
          Kalyani-741235,
          Nadia, West Bengal, India}\par\nopagebreak
  \textit{E-mail address}: \textcolor{blue}{sayan92m@gmail.com}

}}

\begin{document}

\maketitle

\begin{abstract}
\noindent In this article, we will investigate several new configurations
in Ramsey Theory, using the $\varoast_{l,k}$-operation on the set
of integers, recently introduced in \cite{key-4}. This operation
is useful to study symmetric structures in the set of integers, such
as monochromatic configurations of the form $\left\{ x,y,x+y+xy\right\} $
as one of its simplest case. In \cite{key-4}, the author has studied
more general symmetric structures. It has been shown that the Hindman's
Theorem, van der Waerden's Theorem, Deuber's Theorem have their own
symmetric versions. In this article we will explore several new structures,
including polynomial versions of these symmetric structures and some
of its variants. As a result, we get several new symmetric polynomial
configurations as well as new linear symmetric patterns. In the final
section, we will also introduce two new operations on the set of non-negative
integers $\mathbb{N}$, to obtain further new configurations.
\end{abstract}
\subjectclass{05D10}\\

\noindent \keywords{Symmetric configurations, Ramsey theory, Algebra of the Stone-\v{C}ech compactification}

\begingroup 

\color{blue}
\tableofcontents
\endgroup
\vspace{.3in}
\section{Introduction}

Ramsey Theory mainly deals with, ``If a set $A$ has a property,
then how large its subsets need to be to preserve that property?''
One way to do that is to check whether the property holds if the set
$A$ is finitely partitioned or colored. A finite coloring of a set
$A$ is a function from $A$ to a finite set $\left\{ 1,2,\ldots,n\right\} $.
A subset $B$ of $A$ is monochromatic if the function is constant
on $B$. Note that, for a set $X$, $\mathcal{P}\left(X\right)=\left\{ A:A\subseteq X\right\} $
and $\mathcal{P}_{f}\left(X\right)=\left\{ F\subseteq X:F\neq\emptyset\text{ and }F\text{ is finite}\right\} $.

One of the first and basic result of Ramsey Theory is due to van der
Waerden. It proves the partition regularity of arithmetic progression
in the set of natural numbers.
\begin{thm}
\textup{\cite[van der Waerden's Theorem (1927) ]{key-8}} For every
$l\in\mathbb{N}$, and for every finite coloring $\mathbb{N}=C_{1}\cup C_{2}\ldots\cup C_{r}$,
there exists a monochromatic arithmetic progression of length $l$;
that is, there exist a color $C_{i}$ and some elements $a,b$ such
that $a,a+b,a+2b,...,a+lb\in C_{i}$.
\end{thm}

A polynomial version of the above theorem has given by V. Bergelson
and A. Leibman in \cite{key-2}, known as Polynomial van der Waerden's
Theorem.
\begin{thm}
\cite{key-2}\textup{ }Suppose that $p_{1},p_{2},\ldots,p_{m}$ are
polynomials with integer coefficients and no constant term. Then whenever
$\mathbb{N}$ is finitely colored there exist natural numbers $a$
and $d$ such that the point $a$ and all the points $a+p_{i}\left(d\right)$,
for $1\leq i\leq m$, have the same color.
\end{thm}

\subsection{ A brief review of topological algebra}

Now, let us recall some preliminaries of the algebra of the Stone-\v{C}ech
compactification of a discrete semigroup. Let $\left(S,\cdot\right)$
be any discrete semigroup and denote its Stone-\v{C}ech compactification
as $\beta S$. $\beta S$ is the set of all ultrafilters on $S$,
where the points of $S$ are identified with the principal ultrafilters.
The basis for the topology is $\left\{ \bar{A}:A\subseteq S\right\} $,
where $\bar{A}=\left\{ p\in\beta S:A\in p\right\} $. The operation
of $S$ can be extended to $\beta S$ making $\left(\beta S,\cdot\right)$
a compact, right topological semigroup with $S$ contained in its
topological center. That is, for all $p\in\beta S$, the function
$\rho_{p}:\beta S\rightarrow\beta S$ is continuous, where $\rho_{p}\left(q\right)=q\cdot p$
and for all $x\in S$, the function $\lambda_{x}:\beta S\rightarrow\beta S$
is continuous, where $\lambda_{x}\left(q\right)=x\cdot q$. For $p,q\in\beta S$
and $A\subseteq S$, $A\in p\cdot q$ if and only if $\left\{ x\in S:x^{-1}A\in q\right\} \in p$,
where $x^{-1}A=\left\{ y\in S:x\cdot y\in A\right\} $.
\begin{defn}
\cite{key-11} Let $\left(S,\cdot\right)$ be a discrete semigroup.
\end{defn}

\begin{enumerate}
\item The set $A$ is thick if and only if for any finite subset $F$ of
$S$, there exists an element $x\in S$ such that $F\cdot x\subset A$.
This means the sets which contains a translation of any finite subset.
For example, one can see $\cup_{n\in\mathbb{N}}\left[2^{n},2^{n}+n\right]$
is a thick set in $\mathbb{N}$.
\item The set $A$ is syndetic if and only if there exists a finite subset
$G$ of $S$ such that $\bigcup_{t\in G}t^{-1}A=S$. That is, with
a finite translation if, the set which covers the entire semigroup,
then it will be called a Syndetic set. For example, the set of even
and odd numbers are both syndetic in $\mathbb{N}$.
\item The sets which can be written as an intersection of a syndetic and
a thick set are called piecewise syndetic sets. More formally a set
$A$ is piecewise syndetic if and only if there exists $G\in\mathcal{P}_{f}\left(S\right)$
such that for every $F\in\mathcal{P}_{f}\left(S\right)$, there exists
$x\in S$ such that $F\cdot x\subseteq\bigcup_{t\in G}t^{-1}A$. Clearly
the thick sets and syndetic sets are natural examples of piecewise
syndetic sets. From definition one can immediately see that $2\mathbb{N}\cap\bigcup_{n\in\mathbb{N}}\left[2^{n},2^{n}+n\right]$
is a nontrivial example of piecewise syndetic sets in $\mathbb{N}$.
\end{enumerate}
Since $\beta S$ is a compact Hausdorff right topological semigroup,
it has a smallest two sided ideal denoted $K\left(\beta S\right)$,
which is the union of all of the minimal right ideals of $S$, as
well as the union of all of the minimal left ideals of $S$. Every
left ideal of $\beta S$ contains a minimal left ideal and every right
ideal of $\beta S$ contains a minimal right ideal. It can be shown
that a set $A$ is piecewise syndetic if and only if it is a member
of an element of a minimal ultrafilter. The intersection of any minimal
left ideal and any minimal right ideal is a group and any two such
groups are isomorphic. Any idempotent $p$ in $\beta S$ is said to
be minimal if and only if $p\in K\left(\beta S\right)$. A subset
$A$ of $S$ is then central if and only if there is some minimal
idempotent $p$ such that $A\in p$. For more details, the reader
can see \cite{key-11}. In \cite{key-5}, H. Furstenberg defined \textit{central
subsets} of $\mathbb{N}$ in terms of notions from topological dynamics,
showed that if $\mathbb{N}$ is divided into finitely many classes,
then one of these classes must be central, and proved the Central
Sets Theorem.
\begin{thm}
\textup{\cite[The Central Sets Theorem]{key-5} }Let $C$ be a central
subset of $\mathbb{N}$, let $k\in\mathbb{N}$, and for each $i\in\left\{ 1,2,\ldots,k\right\} $,
let $\left\langle y_{i,n}\right\rangle _{n=1}^{\infty}$ be a sequence
in $\mathbb{Z}$. There exist sequences $\left\langle a_{n}\right\rangle _{n=1}^{\infty}$
in $\mathbb{N}$ and $\left\langle H_{n}\right\rangle _{n=1}^{\infty}$
in $\mathcal{P}_{f}\left(\mathbb{N}\right)$ such that 
\begin{enumerate}
\item for each $n$, $\max H_{n}<\min H_{n+1}$ and 
\item for each $i\in\left\{ 1,2,\ldots,k\right\} $ and each $F\in\mathcal{P}_{f}\left(\mathbb{N}\right)$,
\[
\sum_{n\in F}\left(a_{n}+\sum_{t\in H_{n}}y_{i,t}\right)\in C.
\]
\end{enumerate}
\end{thm}

\noindent The Central Sets Theorem can be seen as a joint extension
of both the van der Waerden's Theorem and Hindman's Theorem. It has
an alternative simple algebraic characterization which is in \cite{key-2}.
\begin{defn}
($IP$-sets) Let $\left\langle x_{n}\right\rangle _{n\in\mathbb{N}}$
be an injective sequence in $\mathbb{Z}$. For each $\alpha\in\mathcal{P}_{f}\left(\mathbb{N}\right)$
define $x_{\alpha}=\sum_{n\in\alpha}x_{n}$.

1. The $IP$-set generated by $\left\langle x_{n}\right\rangle _{n\in\mathbb{N}}$
is the set 
\[
\text{FS}\left(\left\langle x_{n}\right\rangle _{n\in\mathbb{N}}\right)=\left\{ x_{\alpha}:\alpha\in\mathcal{P}_{f}\left(\mathbb{N}\right)\right\} .
\]
Clearly, for disjoint $\alpha,\beta\in\mathcal{P}_{f}\left(\mathbb{N}\right)$,
$x_{\alpha\cup\beta}=x_{\alpha}+x_{\beta}$.

2. Let $\left\langle x_{\alpha}\right\rangle _{\alpha\in\mathcal{P}_{f}\left(\mathbb{N}\right)}$,
$\left\langle y_{\alpha}\right\rangle _{\alpha\in\mathcal{P}_{f}\left(\mathbb{N}\right)}$
be $IP$-sets in $\mathbb{Z}$. Now for $\alpha,\beta\in\mathcal{P}_{f}\left(\mathbb{N}\right)$,
$\alpha<\beta$ if and only if $\max_{i\in\alpha}i<\min_{j\in\beta}j$.

3. $\left\langle x_{\alpha}\right\rangle _{\alpha\in\mathcal{P}_{f}\left(\mathbb{N}\right)}$
is called sub-$IP$ set of $\left\langle y_{\alpha}\right\rangle _{\alpha\in\mathcal{P}_{f}\left(\mathbb{N}\right)}$
if there exist $\alpha_{1}<\alpha_{2}<\cdots$ in $\mathcal{P}_{f}\left(\mathbb{N}\right)$
such that $x_{n}=y_{\alpha_{n}}$ for all $n\in\mathbb{N}$.
\end{defn}

\subsection{ Polynomial extension of Deuber's Theorem}

\noindent As a result of his studies about partition regularity,
Deuber \cite{key-2.1} demonstrated further generalizations. In particular,
he showed the partition regularity of the so-called $\left(m,p,c\right)$-sets.
\begin{defn}
\label{m,p,c}Let $m,p,c\in\mathbb{N}$ and let $s=\left(s_{0},\ldots,s_{m}\right)\in\left(\mathbb{Z}\setminus\left\{ 0\right\} \right)^{m+1}$.
The $\left(m,p,c\right)$-set generated by $s$ is the set 
\[
D\left(m,p,c,s\right)=\left\{ \begin{array}{cc}
cs_{0}\\
is_{0}+cs_{1} & i\in\left[-p,p\right]\\
is_{0}+js_{1}+cs_{2} & i,j\in\left[-p,p\right]\\
\vdots & \vdots\\
i_{0}s_{0}+\cdots+i_{m-1}s_{m-1}+cs_{m} & i_{m-1},\ldots,i_{0}\in\left[-p,p\right]
\end{array}\right\} .
\]
\end{defn}

\noindent The following theorem summarizes Deuber\textquoteright s
results from \cite{key-2.1}.
\begin{thm}
For any  $m,p,c\in\mathbb{Z}$ and any finite partition $\mathbb{Z}=\bigcup_{i=1}^{r}C_{i}$,
one of the $C_{i}$ contains an $\left(m,p,c\right)$-set for some
$s\in\left(\mathbb{Z}\setminus\left\{ 0\right\} \right)^{m+1}$.
\end{thm}

The polynomial version of the above theorem has given by V. Bergelson,
J. H. Johnson Jr. and J. Moreira in \cite{key-1}. In the procedure,
they have first proved the Polynomial Central Sets Theorem. In order
to accomplish their results, they have introduced the following terms.
\begin{defn}
\cite{key-1}
\begin{enumerate}
\item \label{R-f}($R$-family) Let $G,H$ be countable commutative semigroups
and let $p\in\beta G$ be an ultrafilter. Let $\Gamma$ be a set of
functions from $H\rightarrow G$. We say that $\Gamma$ is an $R$-family
($R$ stands for returns) with respect to $p$ if for every finite
set $F\subseteq\Gamma$, every $A\in p$ and every IP-set $\left\langle y_{\alpha}\right\rangle _{\alpha\in\mathcal{P}_{f}\left(\mathbb{N}\right)}$
in $H$, there exist $x\in G$ and $\alpha\in F$ such that $x+f\left(y_{\alpha}\right)\in A,\,\forall f\in F$.
\item \label{licit} (Licit) Let $G,H$ be countable commutative semigroups
and let $\Gamma$ be a set of functions from $H$ to $G$. We say
that $\Gamma$ is licit if for any $f\in\Gamma$ and any $z\in H$,
there exists a function $\phi_{z}\in\Gamma$ such that $f\left(y+z\right)=\phi_{z}\left(y\right)+f\left(z\right)$.
\item \label{IP-regular} ($IP$-regular) Let $G$ be a countable commutative
semigroup. An endomorphism $c\in G\rightarrow G$ is called $IP$-regular
if for every $IP$-set $\left\langle x_{\alpha}\right\rangle _{\alpha\in\mathcal{P}_{f}\left(\mathbb{N}\right)}$
in $G$ there exists an IP-set $\left\langle y_{\alpha}\right\rangle _{\alpha\in\mathcal{P}_{f}\left(\mathbb{N}\right)}$
such that $\left\langle c\left(y_{\alpha}\right)\right\rangle _{\alpha\in\mathcal{P}_{f}\left(\mathbb{N}\right)}$
is a sub-IP-set of $\left\langle x_{\alpha}\right\rangle _{\alpha\in\mathcal{P}_{f}\left(\mathbb{N}\right)}$.
\end{enumerate}
\end{defn}

\begin{thm}
\textup{\cite[Polynomial Central Sets Theorem]{key-1} }Let $G,H$
be countable commutative semigroups, let $p\in\beta G$ be an idempotent
ultrafilter, let $\Gamma$ be an $R$-family with respect to $p$
which is licit. Then for any finite set $F\subseteq\Gamma$, any $A\in p$
and any $IP$-set $\left\langle y_{\alpha}\right\rangle _{\alpha\in\mathcal{P}_{f}\left(\mathbb{N}\right)}$
in $H$, there exist a sub-$IP$-set $\left\langle z_{\beta}\right\rangle _{\beta\in\mathcal{P}_{f}\left(\mathbb{N}\right)}$
of $\left\langle y_{\alpha}\right\rangle _{\alpha\in\mathcal{P}_{f}\left(\mathbb{N}\right)}$
and an $IP$-set $\left\langle x_{\beta}\right\rangle _{\beta\in\mathcal{P}_{f}\left(\mathbb{N}\right)}$
in $G$ such that for all $f\in F$ and for all $\beta\in\mathcal{P}_{f}\left(\mathbb{N}\right)$,
\[
x_{\beta}+f\left(y_{\beta}\right)\in A.
\]
\end{thm}

The following definition is the polynomial version of Definition \ref{m,p,c}.
\begin{defn}
Given $m\in\mathbb{N}$, $c:G\rightarrow G$ is a homomorphism, $F$
is an $m$-tuple $F=\left(F_{1},\ldots,F_{m}\right)$ where each $F_{j}$
is a finite set of functions from $G^{j}$ to $G$ and $s=\left(s_{0},\ldots,s_{m}\right)\in\left(G\setminus\left\{ 0\right\} \right)^{m+1}$,
the $\left(m,F,c\right)$-set generated by $s$ is the set 
\[
D\left(m,F,c,s\right)=\left\{ \begin{array}{cc}
c\left(s_{0}\right)\\
f\left(s_{0}\right)+c\left(s_{1}\right) & f\in F_{1}\\
f\left(s_{0},s_{1}\right)+c\left(s_{2}\right) & f\in F_{2}\\
\vdots & \vdots\\
f\left(s_{0},\ldots,s_{m-1}\right)+c\left(s_{m}\right) & f\in F_{m}
\end{array}\right\} .
\]
\end{defn}

The following is the statement of Polynomial Deuber's Theorem.
\begin{thm}
\textup{\cite[Theorem 4.9]{key-1} }Let $G$ be a countable commutative
semigroup, let $p\in\beta G$ be an idempotent ultrafilter and let
$\Gamma_{1},\Gamma_{2},\ldots$ be $R$-families with respect to $p$
which are licit, where $\Gamma_{j}$ consists of maps from $G^{j}$
to $G$. Let $c:G\rightarrow G$ be $IP$-regular, let $m\in\mathbb{N}$
and, for each $j=1,\ldots,m$, let $F_{j}\subseteq\Gamma$ be finite.
Finally, put $F=\left(F_{1},\ldots,F_{m}\right)$. Then for any $A\in p$
there exists an IP-set $\left\langle s_{\alpha}\right\rangle _{\alpha\in\mathcal{P}_{f}\left(\mathbb{N}\right)}$
in $G^{m+1}$ such that $D\left(m,F,c,s_{\alpha}\right)\subseteq A$
for every $\alpha\in\mathcal{F}$.
\end{thm}

\noindent Another basic result of Ramsey theory is N. Hindman's Theorem,
which proves the partition regularity of finite sum of a sequence.
\begin{thm}
\textup{\cite[Hindman's Theorem (1974)]{key-6} }For every finite
coloring $\mathbb{N}=C_{1}\cup C_{2}\ldots\cup C_{r}$, there exists
a sequence $\left\langle x_{n}\right\rangle {}_{n=1}^{\infty}$ in
$\mathbb{N}$and $i\in\left\{ 1,2,\ldots,r\right\} $ such that $FS\left(\left\langle x_{n}\right\rangle {}_{n=1}^{\infty}\right)\subseteq C_{i}$,
where
\[
FS\left(\left\langle x_{n}\right\rangle {}_{n=1}^{\infty}\right)=\left\{ {\displaystyle \sum_{n\in F}x_{n}:F\subseteq\mathcal{P}_{f}\left(\mathbb{N}\right)}\right\} .
\]

Moreover, for every sequence $\left\langle x_{n}\right\rangle {}_{n=1}^{\infty}$
in $\mathbb{N}$ and for any finite coloring $\left\langle x_{n}\right\rangle {}_{n=1}^{\infty}=C_{1}\cup C_{2}\ldots\cup C_{r}$,
there exists a sequence $\left\langle y_{n}\right\rangle {}_{n=1}^{\infty}$
in $\mathbb{N}$ and $i\in\left\{ 1,2,\ldots,r\right\} $ such that
$FS\left(\left\langle y_{n}\right\rangle {}_{n=1}^{\infty}\right)\subseteq C_{i}$.
\end{thm}

\subsection{ Symmetric versions of some Ramsey theoretic results}

In a recent work \cite{key-4}, M. D. Nasso shows the existence of
several monochromatic patterns in the integers obtained as values
of suitable symmetric polynomials, like Brauer's Theorem ( an extention
of van der Waerden's Theorem), Hindman's Theorem, Deuber's Theorem
for symmetric polynomials etc. (see definitions \ref{symm. 1}, \ref{symm. 2}).
\begin{thm}
\textup{\cite[Theorem 2.9]{key-4} }Let $l,k$ be integers where $l\neq0$
divides $k-1$. Then for every finite coloring $\mathbb{Z}=C_{1}\cup\ldots\cup C_{r}$
and for every $L\in\mathbb{N}$ there exist a color $C_{i}$ and elements
$a,b$ such that 
\[
a,b,\frac{1}{l}\left[\left(la+k\right)\left(lb+k\right)-k\right],\ldots,\frac{1}{l}\left[\left(la+k\right)\left(lb+k\right)^{L}-k\right]\in C_{i},
\]
where we can assume elements to be pairwise distinct.

Moreover, for positive $l\in\mathbb{N}$, the above partition regularity
property is also true if we replace the integers $\mathbb{Z}$ with
the natural numbers $\mathbb{N}$.
\end{thm}

{\Large{}}
\begin{thm}
\textup{\cite[Theorem 2.4]{key-4} }Assume that $l,k\neq0$ are integers
where $l$ divides $k\left(k-1\right)$. Then for every finite coloring
$\mathbb{Z}=C_{1}\cup\ldots\cup C_{r}$ there exist an injective sequence
$\left\langle x_{n}\right\rangle _{n=1}^{\infty}$ of integers, and
a color $C_{i}$ such that $\mathfrak{G}_{l,k}\left(x_{n}\right)_{n=1}^{\infty}\subseteq C_{i}$
{[} see Definition \ref{symm. 2} and Definition \ref{seq. sum}{]}.

More generally, for every injective sequence of integers $\left\langle x_{n}\right\rangle _{n=1}^{\infty}$
and for every finite coloring $\mathfrak{G}_{l,k}\left(x_{n}\right)_{n=1}^{\infty}=C_{1}\cup\ldots\cup C_{r}$
of the corresponding $\left(l,k\right)$-symmetric system, there exist
an injective sequence of integers $\left\langle x_{n}\right\rangle _{n=1}^{\infty}$
and a color $C_{i}$ such that $\mathfrak{G}_{l,k}\left(y_{n}\right)_{n=1}^{\infty}\subseteq C_{i}$. 

Moreover, for positive $l\in\mathbb{N}$, the above partition regularity
properties are also true if we replace the integers $\mathbb{Z}$
with the natural numbers $\mathbb{N}$.
\end{thm}

{\Large{}}
\begin{thm}
\textup{\cite[Theorem 2.7]{key-4} }Let $l,k$ be integers where $l\neq0$
divides $k-1$, let $m\in\mathbb{N}$, and let $L\in\mathbb{N}$.
Then for every finite coloring $\mathbb{Z}=C_{1}\cup\ldots\cup C_{r}$
there exist a color $C_{i}$ and elements $a_{0},\ldots,a_{m}\in C_{i}$,
such that for every $j=1,\ldots,m$ and for all $n_{0},\ldots,n_{j-1}\in\left\{ 0,1,\ldots,L\right\} $:
\[
\frac{1}{l}\left(\left(la_{j}+k\right)\prod_{s=0}^{j-1}\left(la_{s}+k\right)^{n_{s}}-k\right)\in C_{i}.
\]
where we can assume that $\left(la_{j}+k\right)\neq0,1,-1$ for all
$j$.\textup{}\\
\end{thm}

\noindent $\bullet$ \textbf{\large{}Organization of this paper:}{\large{}\vspace{.2in}}{\large\par}

\noindent In \textcolor{blue}{Section $2$}, we have recalled the
Hales-Jewett Theorem \cite{key-7} and its polynomial extension \cite{key-9}.
Also, we have looked back on some facts about $\varoast_{l,k}$ operation.\textcolor{blue}{{}
Section $3$} is devoted to study some new additive structures in
piecewise syndetic sets of $\left(\mathbb{N},\varoast_{l,k}\right)$
and some analogues of geo-arithmetic progression of symmetric polynomials.
In this section, we have mainly used the Hales-Jewett Theorem and
a variant of it, given by M. Beiglb\"ock in \cite{key-12}. In \textcolor{blue}{Section
$4$}, an analogue of the polynomial van der Waerden's Theorem has
been derived. This section uses the version of Polynomial Hales-Jewett
Theorem given by M. Walter in \cite{key-9}. Influenced by Section
$4$, we have studied polynomials from $\left(\mathbb{Z},+\right)$
to $\left(\mathbb{Z},\varoast_{l,k}\right)$ explicitly in \textcolor{blue}{Section
$5$}, and provided an analogue of the Central Sets Theorem using
\cite{key-1}. Next in\textcolor{blue}{{} Section $6$}, we deduced
a variant of the Polynomial Deuber's Theorem for polynomials from
$\left(\mathbb{Z}^{i},+\right)$ to $\left(\mathbb{Z},\varoast_{l,k}\right)$,
$i\in\mathbb{N}$, which gives some new examples. At the end, in \textcolor{blue}{Section
$7$}, we will provide two new associate operations over the set of
natural numbers $\mathbb{N}$ to obtain some new configurations, using
Hales-Jewett Theorem.

\section{Preliminaries}

Let us first recall the Hales-Jewett theorem, its polynomial extension
and one of it's variant, which will be necessary throughout our work.

\subsection{ Hales-Jewett theorem}

Let $\omega=\mathbb{N}\cup\left\{ 0\right\} $, where $\mathbb{N}$
is the set of positive integers. Then $\omega$ is the first infinite
ordinal. Given a nonempty set $\mathbb{A}$ (or alphabet) we let any
finite word is of the form $w=a_{1}a_{2}\ldots a_{n}$ with $n\geq1$
and $a_{i}\in\mathbb{A}$. The quantity $n$ is called the length
of $w$ and denoted $\left|w\right|$. Let $v$ (a variable) be a
letter not belonging to $\mathbb{A}$. By a variable word over $\mathbb{A}$
we mean a word $w$ over $\mathbb{A}\cup\left\{ v\right\} $ with
$\left|w\right|_{v}\geq1$. For any variable word $w$, $w\left(a\right)$
is the result of replacing each occurrence of $v$ by $a$.

For any two sets $A,B$ and a function $f:A\rightarrow B$, $\text{Dom}\left(f\right)$
is the domain of the function $f$. A located word $\alpha$ is a
function from a finite set $\text{Dom}\left(\alpha\right)\subseteq\mathbb{N}$
to $\mathbb{A}$. The set of all located words will be denoted by
$L\left(\mathbb{A}\right)$. Note that for located words $\alpha,\beta$
satisfying $\text{Dom}\left(\alpha\right)\cap\text{Dom}\left(\beta\right)=\emptyset$,
$\alpha\cup\beta$ is also located word.

The following theorem is due to A. W. Hales and R. I. Jewett and a
very useful tool to us.
\begin{thm}
\cite[Hales-Jewett Theorem (1963)]{key-7} For all values $t,r\in\mathbb{N}$,
there exists a number $\text{HJ}\left(r,t\right)$ such that, if $N\geq\text{HJ}\left(r,t\right)$
and $\left[t\right]^{N}$ is $r$ colored then there will exists a
monochromatic combinatorial line.
\end{thm}

The word space $\left[t\right]^{N}$ is called Hales-Jewett space
or H-J space. The number $\text{HJ}\left(r,t\right)$ is called Hales-Jewett
number.

Now let us recall the variant of the Hales-Jewett theorem established
by M. Beiglb\"ock in \cite{key-12}.
\begin{thm}
\label{bhj}\cite[Theorem 3]{key-12} Let $\mathcal{F}$ be a partition
regular family of finite subsets of $\mathbb{N}$ which contains no
singletons and let $\mathbb{A}$ be a finite alphabet. For any finite
colouring of $L\left(\mathbb{A}\right)$ there exist $\alpha\in L\left(\mathbb{A}\right),\gamma\in\mathcal{P}_{f}\left(\mathbb{N}\right)$
and $F\in\mathcal{F}$ such that $\text{Dom}\left(\alpha\right),\gamma$
and $F$ are pairwise disjoint and 
\[
\left\{ \alpha\cup\left(\gamma\cup\left\{ t\right\} \right)\times\left\{ s\right\} :s\in\mathbb{A},t\in F\right\} 
\]
is monochromatic.
\end{thm}

The polynomial version of Hales-Jewett Theorem has been first established
in \cite{key-3}, using the methods of topological dynamics and then
it is proved combinatorially by M. Walter in \cite{key-9}.
\begin{thm}
\label{PHJ} \cite[Polynomial Hales-Jewett Theorem]{key-9} For any
$q,k,d$ there exists $N$ such that whenever $Q=Q\left(N\right)=\left[q\right]^{N}\times\left[q\right]^{N\times N}\times\cdots\times\left[q\right]^{N^{d}}$
is $k$-colored there exist $a\in Q$ and $\gamma\subseteq\left[N\right]$
such that the set of points 
\[
\left\{ a\oplus x_{1}\gamma\oplus x_{2}\left(\gamma\times\gamma\right)\oplus\cdots\oplus x_{d}\gamma^{d}:1\leq x_{i}\leq q\right\} 
\]
 is monochromatic.
\end{thm}

The combinatorial configurations we are interested in are symmetric,
in the sense that they originate from suitable symmetric polynomials.
Now, we recall some definitions from \cite{key-4}.

Now we will recall review the $\varoast_{l,k}$ operation.

\subsection{A quick analysis of $\varoast_{l,k}$ operation}

In a recent work \cite{key-4}, M. Di. Nasso in has introduced the
associative operation $\varoast_{l,k}$ on $\mathbb{Z}$ to deduce
many symmetric type configurations in Ramsey theory. The mechanism
of defining this operation is to lift up the multiplicative operation
of the affine space via an isomorphism. Surprisingly $\left(\mathbb{Z},\varoast_{l,k}\right)$
becomes a group. Now we will recall some basic facts from \cite{key-4}
and provide some new observations regarding this operation, which
will be necessary in our work.
\begin{defn}
\label{symm. 1} For $j=1,2,\ldots,n$ the elementary symmetric polynomial
in $n$ variables is the polynomial: 
\[
e_{j}\left(X_{1},X_{2},\ldots,X_{n}\right)=\sum_{1\leq i_{1}\leq\ldots\leq i_{j}\leq n}X_{i_{1}}X_{i_{2}}\cdots X_{i_{j}}=\sum_{\emptyset\neq G\subseteq\left\{ 1,\ldots,n\right\} }\prod_{s\in G}X_{s}.
\]
\end{defn}

For all $a_{1},\ldots,a_{n}$, the product $\prod_{j=1}^{n}\left(a_{j}+1\right)=\sum_{j=1}^{n}e_{j}\left(a_{1},\ldots,a_{n}\right)+1$,
and so 
\[
c=\sum_{j=1}^{n}e_{j}\left(a_{1},\ldots,a_{n}\right)\Longleftrightarrow\prod_{j=1}^{n}\left(a_{j}+1\right)=\left(c+1\right).
\]
More generally, for $l,k\neq0$, it is easily verified that 
\[
\prod_{j=1}^{n}\left(la_{j}+k\right)=lc+k\Longleftrightarrow c=\sum_{j=1}^{n}l^{j-1}k^{n-j}e_{j}\left(a_{1},\ldots,a_{n}\right)+\frac{k^{n}-k}{l}.
\]

Notice the fact that the function $F\left(a_{1},\ldots,a_{n}\right)=c$
where $c$ is the number as above, corresponds to a commutative and
associative operation. Note that the above number $c$ belongs to
$\mathbb{Z}$ for all $a_{1},\ldots,a_{n}\in\mathbb{Z}$ if and only
if $l$ divides $k\left(k-1\right)$.

The function $\mathfrak{G}_{l,k}\left(\cdot\right)$ in next definition
is precisely the same as the value of $c$ in Definition \ref{symm. 1}
and this gives a justification for our attention to this function.
\begin{defn}
\label{symm. 2} For $l,k\in\mathbb{Z}$ with $l,k\neq0$, the $\left(l,k\right)$-symmetric
polynomial in $n$ variables is: 
\[
\mathfrak{G}_{l,k}\left(X_{1},X_{2},\ldots,X_{n}\right)=\sum_{j=1}^{n}l^{j-1}k^{n-j}e_{j}\left(X_{1},X_{2},\ldots,X_{n}\right)+\frac{k^{n}-k}{l}
\]
\[
=\sum_{\emptyset\neq G\subseteq\left\{ 1,\ldots,n\right\} }\left(l^{\left|G\right|-1}k^{n-\left|G\right|}\cdot\prod_{s\in G}X_{s}\right)+\frac{k^{n}-k}{l}.
\]
\end{defn}

In the following definition, the function $\mathfrak{G}_{l,k}\left(\cdot\right)$
is extended to a sequence.
\begin{defn}
\label{seq. sum} Let $\left\langle x_{n}\right\rangle _{n=1}^{\infty}$
be an infinite sequence, and let $l,k\in\mathbb{Z}$ with $l,k\neq0$.
The corresponding $\left(l,k\right)$-symmetric system is the set:
\[
\mathfrak{G}_{l,k}\left(x_{n}\right)_{n=1}^{\infty}=\left\{ \mathfrak{G}_{l,k}\left(x_{n_{1}},x_{n_{2}},\ldots,x_{n_{s}}\right)\mid n_{1}<n_{2}<\cdots<n_{s}\right\} .
\]
For suitable $l$ and $k$, $\left(l,k\right)$-symmetric systems
are partition regular on $\mathbb{Z}$ and on $\mathbb{N}$. 
\end{defn}

For $l,k\in\mathbb{Z}$ where $l\neq0$ divides $k\left(k-1\right)$,
define: 
\[
a\varoast b=c\Longleftrightarrow\left(la+k\right)\left(lb+k\right)=\left(lc+k\right).
\]
So, 
\[
c=a\varoast b=\frac{1}{l}\left[\left(la+k\right)\left(lb+k\right)-k\right]=lab+k\left(a+b\right)+\frac{k^{2}-k}{l}.
\]
Clearly, $c\in\mathbb{Z}$ if and only if $l$ divides $k^{2}-k=k\left(k-1\right)$.

So, we have the following observations.
\begin{fact}
For integers $l\neq0$,
\begin{enumerate}
\item $a\varoast_{l,0}b=lab$;
\item $\mathfrak{G}_{l,0}\left(a_{1},a_{2},\ldots,a_{n}\right)=a_{1}\varoast_{l,0}\cdots\varoast_{l,0}a_{n}=l^{n-1}a_{1}\cdots a_{n}$;
\end{enumerate}
\end{fact}

The iterated $\varoast_{l,k}$-products are exactly the function defined
in Definition \ref{symm. 2} and to verify, we recall the following
propositions from \cite{key-4}.
\begin{prop}
Let $l,k\in\mathbb{Z}$ be such that $l\neq0$ divides $k\left(k-1\right)$.
Then for all $a_{1},\ldots,a_{n}\in\mathbb{Z}$:
\begin{align*}
a_{1}\varoast_{l,k}\cdots\varoast_{l,k}a_{n} & =\sum_{\emptyset\neq G\subseteq\left\{ 1,\ldots,n\right\} }\left(l^{\left|G\right|-1}k^{n-\left|G\right|}\cdot\prod_{s\in G}a_{s}\right)+\frac{k^{n}-k}{l}\\
 & =\sum_{j=1}^{n}l^{j-1}k^{n-j}e_{j}\left(a_{1},a_{2},\ldots,a_{n}\right)+\frac{k^{n}-k}{l}\\
 & =\mathfrak{G}_{l,k}\left(a_{1},a_{2},\ldots,a_{n}\right).
\end{align*}
\end{prop}

The following property of the function $\mathfrak{G}_{l,k}\left(\cdot\right)$
are useful for us.
\begin{fact}
\textbf{Some observation about $\mathfrak{G}_{l,k}\left(\cdot\right)$:}

1. For any permutation $f$ of $\left\{ a_{1},a_{2},\ldots,a_{n}\right\} $,
as $\mathfrak{G}_{l,k}\left(\cdot\right)$ is a symmetric function,
$\mathfrak{G}_{l,k}\left(a_{1},a_{2},\ldots,a_{n}\right)=\mathfrak{G}_{l,k}\left(a_{f\left(1\right)},a_{f\left(2\right)},\ldots,a_{f\left(n\right)}\right)$.
\begin{proof}
From the definition of \textbf{$\mathfrak{G}_{l,k}\left(\cdot\right)$},
we have
\begin{align*}
\mathfrak{G}_{l,k}\left(a_{1},a_{2},\ldots,a_{n}\right) & =\sum_{j=1}^{n}l^{j-1}k^{n-j}e_{j}\left(a_{1},a_{2},\ldots,a_{n}\right)+\frac{k^{n}-k}{l}\\
 & =\sum_{j=1}^{n}l^{j-1}k^{n-j}e_{j}\left(a_{f\left(1\right)},a_{f\left(2\right)},\ldots,a_{f\left(n\right)}\right)+\frac{k^{n}-k}{l}\\
 & =\mathfrak{G}_{l,k}\left(a_{f\left(1\right)},a_{f\left(2\right)},\ldots,a_{f\left(n\right)}\right).
\end{align*}
\end{proof}
2. $\mathfrak{G}_{l,k}\left(a,b\right)\varoast_{l,k}\mathfrak{G}_{l,k}\left(c,d,e\right)=a\varoast_{l,k}b\varoast_{l,k}c\varoast_{l,k}d\varoast_{l,k}e=\mathfrak{G}_{l,k}\left(a,b,c,d,e\right)$.
\end{fact}

Now, we define some notations we will use.
\begin{defn}
Let $\left\langle a_{n}\right\rangle _{n=1}^{\infty}$ be an injective
sequence in $\mathbb{Z}$, and $\left\langle H_{n}\right\rangle _{n=1}^{\infty}$
be a sequence in $\mathcal{P}_{f}\left(\mathbb{N}\right)$, also let
$B=\left\{ x_{n}:n\in\mathbb{N}\right\} $ then we have the followings.
\begin{enumerate}
\item For $\alpha=\left\{ i_{1},i_{2},\ldots,i_{m}\right\} \subseteq\mathbb{N}$
satisfying $i_{1}<i_{2}<\cdots<i_{m}$. Then, 
\[
a_{\alpha}^{<}=\left\{ a_{i_{1}},a_{i_{2}},\ldots,a_{i_{m}}\right\} \text{ and }a_{\alpha}=\sum_{i\in\alpha}a_{i}.
\]
\item For $\alpha\in\mathcal{P}_{f}\left(\mathbb{N}\right)$ we have $H_{\alpha}=\bigcup_{i\in\alpha}H_{i}$.
\item Let $\text{FS}\left\langle y_{n}\right\rangle _{n=1}^{\infty}$ be
a sub-$IP$ set of $\text{FS}\left\langle x_{n}\right\rangle _{n=1}^{\infty}$.
Then, for any given $\beta=\left\{ i_{1},i_{2},\ldots,i_{p}\right\} \in\mathcal{P}_{f}\left(\mathbb{N}\right)$,
we define 
\[
y_{\beta}^{<\left(B\right)}=\left(x_{i_{1}},x_{i_{2}},\ldots,x_{i_{p}}\right)\,\text{and}\,y_{\beta}=\sum_{j=1}^{p}x_{i_{j}}.
\]
 We call the sequence $B$ is a \textit{base sequence.}
\end{enumerate}
\end{defn}

\noindent The following remark is a useful accessory to our work.
\begin{rem}
\label{inv.} For any $a,b,d\in\mathbb{Z}$ and for $i\in\mathbb{N}$,
\[
b\varoast_{l,k}\left(a+id\right)
\]

\[
\qquad\qquad\qquad\qquad=\frac{1}{l}\left[\left(lb+k\right)\left(l\left(a+id\right)+k\right)-k\right]
\]

\[
\qquad\qquad\qquad\qquad\quad\qquad\qquad=l\left(ab+bk+ak\right)+\frac{k^{2}-k}{l}+i\left(l^{2}bd+kld\right)
\]

\[
\!\!\!\!\!\!\!\!\!\!\!\!=y+iz,
\]
where $y=l\left(ab+bk+ak\right)+\frac{k^{2}-k}{l}$, and $z=l^{2}bd+kld$.
\end{rem}

\section{Some applications of Hales-Jewett Theorem}

In the first subsection, we will study two variations of geo-arithmetic
progressions and later we will show some additively rich structures
in the piecewise syndetic sets of $\left(\mathbb{Z},\varoast_{l,k}\right)$.

\subsection{ Two variants of Geo-arithmetic progressions}

In \cite{key-13}, the authors have proved that for any finite partition
of $\mathbb{Z}$, there exists a cell which contains geo-arithmetic
progressions. Formally,
\begin{thm}
If $n,r\in\mathbb{N}$ and $\mathbb{Z}=C_{1}\cup C_{2}\cup\ldots\cup C_{r},$
then there exists $k\in\left\{ 1,2,\ldots,r\right\} $ and $a,b,d\in\mathbb{Z}$,
such that 
\[
\left\{ b\cdot\left(a+i\cdot d\right)^{j}:0\leq i,j\leq n\right\} \subset C_{k}.
\]
\end{thm}

Now using the Theorem \ref{bhj}, we will deduce two versions of geo-arithmetic
progressions. To do so, we will use the van der Waerden's theorem
\cite{key-8} in our first proof and its variant for symmetric polynomials
\cite[Theorem 2.9]{key-4} in our second proof. Let us now explore
the symmetric version of monochromatic geo-arithmatic preogressions
in $\mathbb{Z}$, which involves the addition ``$+$'' and ``$\varoast_{l,k}$''
operations on $\mathbb{Z}$.
\begin{thm}
\label{1geo} Let $l,k\in\mathbb{Z}$ and $\mathbb{Z}=\bigcup_{i=1}^{r}C_{i}$
be any partition of $\mathbb{Z}$. Then for each $m\in\mathbb{N}$,
there exist $x,y,z\in\mathbb{Z}$ such that the following configuration
\[
\left\{ \frac{1}{l}\left[\left(lx+k\right)\left(l\left(y+iz\right)+k\right)^{j}-k\right]:i,j\in\left\{ 0,1,\ldots,m\right\} \right\} ,
\]
 is monochromatic. 
\end{thm}

\begin{proof}
Let $\mathcal{F}=\left\{ \left\{ a,a+d,\ldots,a+kd\right\} :a,d\in\mathbb{N}\right\} $
be the set of all $\left(k+1\right)$-term arithmetic progressions.
Clearly $\mathcal{F}$ is a partition regular family over $\mathbb{Z}$.
Take $\mathbb{A}=\left\{ 0,1,\ldots,m\right\} $ and define $f:L\left(\mathbb{A}\right)\rightarrow\mathbb{N}$
by $f\left(\alpha\right)=\prod_{t\in\text{Dom}\left(\alpha\right)}t^{\left(\alpha\left(t\right)\right)}$.
Color each $\alpha\in L\left(\mathbb{A}\right)$ with the color of
$f\left(\alpha\right)$. 

Now, choose $\alpha,\gamma\in L\left(\mathbb{A}\right)$ and $F\in\mathcal{F}$
as in the Theorem \ref{bhj}. Then for all $i,j\in\left\{ 0,1,\ldots,k\right\} $,
the following configuration is monochromatic.

$f\left(\alpha\cup\left(\gamma\cup\left\{ a+id\right\} \right)\times\left\{ j\right\} \right)$

$=\prod_{t\in\text{Dom}\left(\alpha\right)}t^{\left(\alpha\left(t\right)\right)}\varoast_{l,k}\prod_{t\in\gamma}t^{\left(j\right)}\varoast_{l,k}\left(a+id\right)^{\left(j\right)}$

$=x\varoast_{l,k}b^{\left(j\right)}\varoast_{l,k}\left(a+id\right)^{\left(j\right)}$,
where $b=\prod_{t\in\gamma}t$.

$=x\varoast_{l,k}\left(b\varoast_{l,k}\left(a+id\right)\right)^{\left(j\right)}$

$=x\varoast_{l,k}\left(y+iz\right)^{\left(j\right)}$, for some $y,z\in\mathbb{Z}$
by Remark \ref{inv.}.

Now $c=x\varoast_{l,k}\left(y+iz\right)^{\left(j\right)}\iff c=\frac{1}{l}\left[\left(lx+k\right)\left(l\left(y+iz\right)+k\right)^{j}-k\right].$
\end{proof}
The above version has it's own interest. Suppose $\left(l,k\right)=\left(1,0\right)$,
then theorem \ref{1geo} gives the monochromaticity of the ordinary
geo-arthmetic progression on $\mathbb{Z}$ as a special case. The following
version of geo-arithmatic progressions in $\mathbb{Z}$ involves only
the ``$\varoast_{l,k}$'' operations.
\begin{thm}
Let $l,k\in\mathbb{Z}$ and $\mathbb{Z}=\bigcup_{i=1}^{r}C_{i}$ be
any partition of $\mathbb{Z}$. Then for each $m\in\mathbb{N}$, there
exist $x,y,z\in\mathbb{Z}$ such that the following configuration
\[
\left\{ \frac{1}{l}\left[\left(lx+k\right)\left(ly+k\right)^{j}\left(lz+k\right)^{ij}-k\right]:i,j\in\left\{ 0,1,\ldots,m\right\} \right\} ,
\]
 is monochromatic. 
\end{thm}

\begin{proof}
Consider $\mathcal{F}=\left\{ \left\{ a,a\varoast_{l,k}d,\ldots,a\varoast_{l,k}d^{\left(k\right)}\right\} :a,d\in\mathbb{N}\right\} $.
Clearly $\mathcal{F}$ is a partition regular family over $\mathbb{Z}$.
Take $\mathbb{A}=\left\{ 0,1,\ldots,m\right\} $ and define $f:L\left(\mathbb{A}\right)\rightarrow\mathbb{N}$
by $f\left(\alpha\right)=\prod_{t\in\text{Dom}\left(\alpha\right)}t^{\left(\alpha\left(t\right)\right)}$.

Now, choose $\alpha,\gamma\in L\left(\mathbb{A}\right)$ and $F\in\mathcal{F}$
as in the Theorem \ref{bhj}. Then for all $i,j\in\left\{ 0,1,\ldots,k\right\} $,
the following configuration is monochromatic.

$f\left(\alpha\cup\left(\gamma\cup\left\{ a\varoast_{l,k}d^{\left(i\right)}\right\} \right)\times\left\{ j\right\} \right)$

$=\prod_{t\in\text{Dom}\left(\alpha\right)}t^{\left(\alpha\left(t\right)\right)}\varoast_{l,k}\prod_{t\in\gamma}t^{\left(j\right)}\varoast_{l,k}\left(a\varoast_{l,k}d^{\left(i\right)}\right)^{\left(j\right)}$

$=x\varoast_{l,k}c^{\left(j\right)}\varoast_{l,k}\left(a\varoast_{l,k}d^{\left(i\right)}\right)^{\left(j\right)}$,
where $c=\prod_{t\in\gamma}t$.

$=x\varoast_{l,k}\left(a\varoast_{l,k}c\varoast_{l,k}d^{\left(i\right)}\right)^{\left(j\right)}$

$=x\varoast_{l,k}\left(b\varoast_{l,k}d^{\left(i\right)}\right)^{\left(j\right)}$,
where $b=a\varoast_{l,k}c$.

Replacing $b$ by $y$ and $d$ by $z$, we get,

$c=x\varoast_{l,k}\left(b\varoast_{l,k}d^{\left(i\right)}\right)^{\left(j\right)}\iff c=\frac{1}{l}\left[\left(lx+k\right)\left(ly+k\right)^{j}\left(lz+k\right)^{ij}-k\right]$.
\end{proof}

\subsection{ Additive structure in $\left(\mathbb{N},\varoast_{l,k}\right)$}

Note that any piecewise syndetic set in $\left(\mathbb{N},+\right)$
contains arithmetic progressions of arbitrary length. In fact it contains
a more general structure, the generalized arithmetic progressions.
Now we will study whether a piecewise syndetic set in $\left(\mathbb{N},\varoast_{l,k}\right)$
contains arithmetic progressions of $\left(\mathbb{N},+\right)$ of
arbitrary length or not. Note that, for $m,n,r\in\mathbb{N}$ if we
take any $r$-partition of $\mathbb{N}$, we have a monochromatic configuration
of the form 
\[
\left\{ a_{0}+i_{1}a_{1}+\cdots+i_{m}a_{m}:i_{1},\ldots,i_{m}\in\left\{ 0,1,\ldots,n\right\} \right\} 
\]
 called generalized arithmetic progression of length $n$, order $m$.

Now, for any $x\in\mathbb{N}$,

$x\varoast_{l,k}\left(a_{0}+i_{1}a_{1}+\cdots+i_{m}a_{m}\right)$

$=\frac{1}{l}\left[\left(lx+k\right)\left(l\left(a_{0}+i_{1}a_{1}+\cdots+i_{m}a_{m}\right)+k\right)-k\right]$

$=la_{0}x+kx+ka_{0}+\frac{k^{2}-k}{l}+\sum_{j=1}^{m}i_{j}\left(la_{j}x+ka_{j}\right)$

$=P_{0}+i_{1}P_{1}+i_{2}P_{2}+\cdots+i_{m}P_{m}$(say).

Hence the family of generalized arithmetic progression of length $n$,
order $m$ is an invariant partition regular family of $\left(\mathbb{N},\varoast_{l,k}\right)$.
So, from \cite[Lemma 1]{key-12}, we have the following result;
\begin{thm}
Let $A\subseteq\left(\mathbb{N},\varoast_{l,k}\right)$ be a piecewise
syndetic set and $m,n\in\mathbb{N}$. Then there exist $a_{0},a_{1},\ldots,a_{m}\in\mathbb{N}$
such that 
\[
\left\{ a_{0}+i_{1}a_{1}+\cdots+i_{m}a_{m}:i_{1},\ldots,i_{m}\in\left\{ 0,1,\ldots,n\right\} \right\} \subseteq A.
\]
\end{thm}

Now, let us recall from \cite{key-2} that any piecewise syndetic
set in $\left(\mathbb{N},+\right)$ contains polynomial progressions.
Now we will show that a weak version of polynomial progressions, where
polynomials are considered from $\left(\mathbb{Z},+\right)$ to $\left(\mathbb{Z},+\right)$,
are contained in piecewise syndetic sets in $A\subseteq\left(\mathbb{N},\varoast_{l,k}\right)$
\begin{thm}
Let $P_{1},P_{2},\ldots,P_{m}$ be a finite set of polynomials defined
on $\left(\mathbb{N},+\right)$ with zero constant term. Then for
any piecewise syndetic set $A\subseteq\left(\mathbb{N},\varoast_{l,k}\right)$,
there exist $a,b,d\in\mathbb{N}$ such that 
\[
\left\{ a+bP_{i}\left(d\right):1\leq i\leq m\right\} \subseteq A.
\]
\end{thm}

\begin{proof}
Take $\mathcal{F}=\left\{ \left\{ a+bP_{i}\left(d\right)\right\} _{i=1}^{m}:a,b,d\in\mathbb{N}\right\} $.
Then, $\mathcal{F}$ is partition regular over $\mathbb{N}$.

Now for any $x\in\mathbb{N},$

$x\varoast_{l,k}\left(a+bP_{i}\left(d\right)\right)$

$=\frac{1}{l}\left[\left(lx+k\right)\left(l\left(a+bP_{i}\left(d\right)\right)+k\right)-k\right]$

$=\left(lax+kx+ak+\frac{k^{2}-k}{l}\right)+\left(lx+k\right)bP_{i}\left(d\right)$

$=p+qP_{i}\left(d\right)$, where $p=lax+kx+ak+\frac{k^{2}-k}{l}$
and $q=\left(lx+k\right)b$.

So, $\mathcal{F}$ is invariant partition regular family over $\left(\mathbb{N},\varoast_{l,k}\right)$.
So, using \cite[Lemma 1]{key-12}, we conclude the proof.
\end{proof}
\begin{example}
Let us take two polynomials $p_{1}\left(x\right)=2x$ and $p_{2}\left(x\right)=x^{2}$,
then from the above theorem there exist $a,b,c\in\mathbb{N}$ such
that any piecewise syndetic set in $\left(\mathbb{N},\varoast_{l,k}\right)$
contains configurations of the form
\[
\left\{ a+p_{1}\left(b\right)c,a+p_{2}\left(b\right)c\right\} =\left\{ a+2bc,\,a+b^{2}c\right\} .
\]
\end{example}

\begin{example}
Let us take three polynomials $p_{1}\left(x\right)=2x$, $p_{2}\left(x\right)=3x^{2}$
and $p_{3}\left(x\right)=4x^{3}$,then from the above theorem there
exist $a,b,c\in\mathbb{N}$ such that any piecewise syndetic set in
$\left(\mathbb{N},\varoast_{l,k}\right)$ contains configurations
of the form
\[
\left\{ a+2bc,\,a+3b^{2}c,\,a+4b^{3}c\right\} .
\]
\end{example}

\section{An analogue to Polynomial van der Waerden's theorem}

In the suitable symmetric polynomial setting, our analogue for Polynomial
van der Waerden's Theorem is the following:
\begin{thm}
\label{pvdw} Let $d\in\mathbb{N}$ and $\left\{ a_{1}^{\left(i\right)},a_{2}^{\left(i\right)},\ldots,a_{d}^{\left(i\right)}\right\} _{i=1}^{m}\subseteq\mathbb{Z}\setminus\left\{ -\frac{k}{l},-\frac{k+1}{l}\right\} $
for any $m\in\mathbb{N}$. Then there exist $d',c\in\mathbb{N}$
such that 
\[
\frac{1}{l}\left[\left(ld'+k\right)\left(la_{1}^{\left(i\right)}+k\right)^{c}\left(la_{2}^{\left(i\right)}+k\right)^{c^{2}}\ldots\left(la_{d}^{\left(i\right)}+k\right)^{c^{d}}-k\right]
\]
 is monochromatic.
\end{thm}

\begin{fact}
\textbf{Deducing Symmetric van dar Waerden's Theorem from Theorem
\ref{pvdw}:} 

\textup{Now, To verify that the above theorem is really giving the
polynomial version of van der Waerden's Theorem for symmetric polynomial,
let, $\mathbb{Z}=\bigcup_{i=1}^{r}C_{i}$, and $d=1$. }

\textup{Let, for $1\leq i\leq m$, $a^{\left(i\right)}=\frac{1}{l}\left[\left(la^{\left(1\right)}+k\right)^{i}-k\right]$,
and $a^{\left(m+1\right)}=-\frac{k-1}{l}$.}

\textup{Then, there exist $d',c$ such that,
\[
\frac{1}{l}\left[\left(ld'+k\right)\left(la^{\left(i\right)}+k\right)^{c}-k\right],1\leq i\leq m+1
\]
 are monochromatic.}

\textup{Let, $\left(la^{\left(1\right)}+k\right)^{c}=lx+k$, then
\begin{align*}
\frac{1}{l}\left[\left(ld'+k\right)\left(lx+k\right)-k\right],\frac{1}{l}\left[\left(ld'+k\right)\left(lx+k\right)^{2}-k\right],\\
\ldots,\frac{1}{l}\left[\left(ld'+k\right)\left(lx+k\right)^{m}-k\right],d'
\end{align*}
 are monochromatic, which proves our claim.}
\end{fact}

Now, here are two examples which show in a simple way the type of
configurations are monochromatic in the Theorem \ref{pvdw}.
\begin{example}
Let $\left(l,k\right)=\left(3,1\right)$ and choose any $n\in\mathbb{N}$.
Let the finite sequence $\left\langle a^{\left(i\right)}\right\rangle _{i=1}^{n}$
is defined by,

$a^{\left(1\right)}=\left(a_{1}^{\left(1\right)},a_{2}^{\left(1\right)},\ldots,a_{n}^{\left(1\right)}\right)=\left(1,0,\ldots,0\right)$, 

$a^{\left(2\right)}=\left(a_{1}^{\left(2\right)},a_{2}^{\left(2\right)},\ldots,a_{n}^{\left(2\right)}\right)=\left(0,1,\ldots,0\right)$
and so on. 

In general, for $i\in\left\{ 1,2,\ldots,n\right\} $, $a^{\left(i\right)}=\left(0,0,\ldots,1,\ldots,0\right)$,
where $1$ is at the $i^{th}$ coordinate.

Then, by Theorem \ref{pvdw}, there exist $x,y\in\mathbb{N}$ such
that 
\[
\left\{ \frac{1}{3}\left(3x+1\right)4^{y}-1,\frac{1}{3}\left(3x+1\right)4^{y^{2}}-1,\ldots,\frac{1}{3}\left(3x+1\right)4^{y^{n}}-1\right\} 
\]
 is monochromatic.
\end{example}

\begin{example}
Let $\left(l,k\right)=\left(2,1\right)$ and choose any $n\in\mathbb{N}$.
Now, take the finite sequence $a^{\left(1\right)}=\left(1,0,\ldots,0\right),a^{\left(2\right)}=\left(3,0,\ldots,0\right),...,a^{\left(n\right)}=\left(2n+1,0,\ldots,0\right)$
and so by Theorem \ref{pvdw}, there exists $x,c$ such that 

\[
\left\{ \frac{1}{2}\left(2x+1\right)3^{c}-1,\frac{1}{2}\left(2x+1\right)5^{c}-1,\ldots,\frac{1}{2}\left(2x+1\right)\left(2n+1\right)^{c}-1\right\} 
\]
 is monochromatic.
\end{example}

Now, we will prove the theorem \ref{pvdw}. The following proof uses
the polynomial Hales-Jewett theorem \cite{key-9}.
\begin{proof}[\textbf{\textit{Proof of Theorem \ref{pvdw}}}]
 Let $q=\left\{ a_{1}^{\left(i\right)},a_{2}^{\left(i\right)},\ldots,a_{m}^{\left(i\right)}\right\} _{i=1}^{d}$
and choose  $N=\text{PHJ}\left(q,r,d\right)$. Then by Theorem \ref{PHJ},
one cell of the $r$-colored partition of $Q\left(N\right)=\left[q\right]^{N}\times\left[q\right]^{N\times N}\times\ldots\times\left[q\right]^{N^{d}}$
will contain a set of the form 
\[
\left\{ a\oplus x_{1}\gamma\oplus x_{2}\left(\gamma\times\gamma\right)\oplus\ldots\oplus x_{d}\gamma^{d}:1\leq x_{i}\leq q\right\} .
\]

Now, our $r$-coloring on $\mathbb{Z}\setminus\left\{ -\frac{k}{l},-\frac{k+1}{l}\right\} $
induces a coloring on $\left[q\right]^{N}\times\left[q\right]^{N\times N}\times\ldots\times\left[q\right]^{N^{d}}$
by taking each string in the Hales-Jewett space say $a_{1}a_{2}\ldots a_{R}\rightarrow a_{1}\varoast_{l,k}a_{2}\varoast_{l,k}\cdots\varoast_{l,k}a_{R}$.

Now, let $A\subseteq\mathbb{Z}$ is piecewise syndetic. So, there
exists a finite set $E$ such that $E^{-1}A$ is thick. As the image
of $Q\left(N\right)$ is finite, translating it by an element say,
$t\in\mathbb{Z}\setminus\left\{ -\frac{k}{l},-\frac{k+1}{l}\right\} $
(we can choose such $t$) such that $\text{IM}\left(Q\left(N\right)\right)\subseteq E^{-1}A$.
Now, give an $r$-color on $Q\left(N\right)$ as for $x,y\in Q\left(N\right)$,
$x\sim y$ if and only if $\text{Im}\left(x\right)$,$\text{Im}\left(y\right)\in t_{1}^{-1}A$
for some $t_{1}\in E$. 

So, we have a monochromatic combinatorial line 
\[
\left\{ a\oplus x_{1}\gamma\oplus x_{2}\left(\gamma\times\gamma\right)\oplus\ldots\oplus x_{d}\gamma^{d}:1\leq x_{i}\leq q;\,1\leq i\leq d\right\} .
\]
 Now, each $a\oplus x_{1}\gamma\oplus x_{2}\left(\gamma\times\gamma\right)\oplus\ldots\oplus x_{d}\gamma^{d}$
goes to 
\[
t_{1}\varoast_{l,k}b_{1}\varoast_{l,k}b_{2}\varoast_{l,k}\ldots\varoast_{l,k}b_{s}\varoast_{l,k}x_{1}^{\left(\left|\gamma\right|\right)}\varoast_{l,k}x_{2}^{\left(\left|\gamma\right|^{2}\right)}\varoast_{l,k}\ldots\varoast_{l,k}x_{d}^{\left(\left|\gamma\right|^{d}\right)},
\]
 where $x_{i}\in\left[q\right]$ for all $i\in\left\{ 1,2,\ldots,d\right\} $.

All of the configurations are in $A$, where without loss of generality,
we have taken $t_{1}\in E$.

Let, $d=t_{1}\varoast_{l,k}t\varoast_{l,k}b_{1}\varoast_{l,k}b_{2}\varoast_{l,k}\ldots\varoast_{l,k}b_{s}$
and therefore the other configurations becomes $d\varoast_{l,k}x_{1}^{\left(\left|\gamma\right|\right)}\varoast_{l,k}x_{2}^{\left(\left|\gamma\right|^{2}\right)}\varoast_{l,k}\ldots\varoast_{l,k}x_{d}^{\left(\left|\gamma\right|^{d}\right)}$
where $x_{i}\in\left[q\right]$ for all $i\in\left\{ 1,2,\ldots,d\right\} $.

Now, $d\varoast_{l,k}x_{1}^{\left(\left|\gamma\right|\right)}\varoast_{l,k}x_{2}^{\left(\left|\gamma\right|^{2}\right)}\varoast_{l,k}\ldots\varoast_{l,k}x_{d}^{\left(\left|\gamma\right|^{d}\right)}$

$=\,\,\frac{1}{l}\left(\left(ld'+k\right)\left(lx_{1}+k\right)^{c}\left(lx_{2}+k\right)^{c^{2}}\ldots\left(lx_{d}+k\right)^{c^{d}}-k\right)$
where $x_{i}\in\left[q\right]$ for all $i\in\left\{ 1,2,\ldots,d\right\} $.

This proves the theorem.
\end{proof}

\section{ Symmetric Polynomial Central Sets Theorem}

In \cite{key-4}, the author used homomorphism map from $\left(\mathbb{Z}.\varoast_{l,k}\right)$
to $\left(\mathbb{Z}.\varoast_{l,k}\right)$ to provided various new
Ramsey theoretic configurations. Here, we have studied polynomial,
i.e; polynomials from $\left(\mathbb{Z},+\right)$ to $\left(\mathbb{Z},\varoast_{l,k}\right)$.

We know any polynomial on the set of natural numbers $\mathbb{N}$
is of the form: 
\[
P\left(x\right)=a_{n}x^{n}+a_{n-1}x^{n-1}+\ldots+a_{0},\,n\in\mathbb{N}.
\]

Now, $P\left(x\right)=a_{n}x^{n}+a_{n-1}x^{n-1}+\ldots+a_{0}$

$\ \ \ \,\,\,\ \ =\left[a_{n}+\ldots+a_{n}\right]\,\left(x^{n}\text{-times}\right)+\left[a_{n-1}+\ldots+a_{n-1}\right]\,\left(x^{n-1}\text{-times}\right)+\ldots+a_{0}$.

Hence, it is natural to expect a polynomial $P:\left(\mathbb{N},+\right)\rightarrow\left(\mathbb{Z}.\varoast_{l,k}\right)$
is of the form,

\[
P\left(x\right)=a_{n}^{\left(x^{n}\right)}\varoast_{l,k}a_{n-1}^{\left(x^{n-1}\right)}\varoast_{l,k}\cdots\varoast_{l,k}a_{0}.
\]
 But we want to extend the domain from $\left(\mathbb{N},+\right)$
to $\left(\mathbb{Z},+\right)$. Interestingly, it is possible. First,
note that for polynomials from $\left(\mathbb{Z},+\right)$ to $\left(\mathbb{Z},+\right)$,

\[
a_{n}x^{n}=\left\{ \begin{array}{c}
a_{n}+a_{n}+\ldots+a_{n}\,\left(x^{n}\text{-times}\right),\ \text{if }n\in2\mathbb{Z}\\
b_{n}+b_{n}+\ldots+b_{n}\,\left(\left|x\right|^{n}\text{-times}\right),\ \text{if }n\in2\mathbb{Z}+1
\end{array}\right.,
\]
 where $b_{n}=\left(\text{sgn }x\right)a_{n}$.

Now, if $c\in\mathbb{Z}$, define,

\[
a^{\left(c\right)}=\left\{ \begin{array}{cc}
a\varoast_{l,k}a\varoast_{l,k}\cdots\varoast_{l,k}a\,\left(c\text{-times}\right) & \text{if }c>0\\
-\frac{k-1}{l} & \text{if }c=0\\
a^{-1}\varoast_{l,k}a^{-1}\varoast_{l,k}\cdots\varoast_{l,k}a^{-1}\,\left(\left|c\right|\text{-times}\right) & \text{if }c<0
\end{array}\right..
\]
 Or, in short $a^{\left(c\right)}=\left(a^{\left(\text{sgn }x\right)}\right)^{\left(\left|c\right|\right)}$,
if $c\neq0$ and $-\frac{k-1}{l}$, if $c=0$.

Note that, $a^{\left(c+d\right)}=a^{\left(c\right)}\cdot a^{\left(d\right)}$
and $a^{\left(cd\right)}=\left(a^{\left(c\right)}\right)^{\left(d\right)}$
for $c,d\in\mathbb{Z}$.

Now, for $n\in\mathbb{N}$, a polynomial $P:\left(\mathbb{Z},+\right)\rightarrow\left(\mathbb{Z},\varoast_{l,k}\right)$
is defined by,

\[
P\left(x\right)=a_{n}^{\left(x^{n}\right)}\varoast_{l,k}a_{n-1}^{\left(x^{n-1}\right)}\varoast_{l,k}\cdots\varoast_{l,k}a_{0},
\]
 where $a_{n},a_{n-1},\ldots,a_{0}\in\mathbb{Z}$, $a_{n}$ is the
leading coefficient and $a_{0}$ is the constant term of the polynomial.

So, $y=P\left(x\right)\Leftrightarrow y=\mathfrak{G}_{l,k}\left(a_{n}^{\left(x^{n}\right)},a_{n-1}^{\left(x^{n-1}\right)},\ldots,a_{o}\right)$

$\Leftrightarrow y=\frac{1}{l}\left[\left(la_{n}+k\right)^{x^{n}}\left(la_{n-1}+k\right)^{x^{n-1}}\cdots\left(la_{0}+k\right)-k\right]$.

From Polynomial van der Waerden's Theorem, we have seen that for any
finite set of zero constant polynomials, we have a monochromatic polynomial
progression. 

Now, we will verify that our polynomial is a group polynomial defined
in \cite{key-1}.

Note for any constant $a\in\mathbb{Z}$, is a polynomial of degree
$0$, let $f$ be a polynomial of degree $d$, then for $h\in\mathbb{Z}$,

$f\left(x+h\right)f\left(x\right)^{-1}=\left(a_{n}^{\left(x+h\right)^{n}}\varoast\left(a_{n}^{-1}\right)^{x^{n}}\right)\varoast_{l,k}\cdots\varoast_{l,k}\left(a_{1}^{\left(x+h\right)}\varoast_{l,k}\left(a_{1}^{-1}\right)^{\left(x\right)}\right)$

$=\prod_{i=1}^{n}a_{i}^{\left(x+h\right)^{i}}\varoast\left(a_{i}^{-1}\right)^{x^{i}}$

$=\prod_{i=1}^{n}a_{i}^{\left(x^{i}+{i \choose 1}x^{i-1}h+\cdots+h^{i}\right)}\varoast_{l,k}\left(a_{i}^{-1}\right)^{\left(x^{i}\right)}$

$=\prod_{i=1}^{n}a_{i}^{\left(x^{i}\right)}\varoast_{l,k}\left(a_{i}^{-1}\right)^{\left(x^{i}\right)}\varoast_{l,k}a_{i}^{\left(\sum_{j=1}^{i}{i \choose j}x^{i-j}h^{j}\right)}$

$=\prod_{i=1}^{n}a_{i}^{\left(\sum_{j=1}^{i}{i \choose j}x^{i-j}h^{j}\right)}$

$=\prod_{i=1}^{n}\left[\prod_{j=1}^{i}a_{i}^{\left({i \choose j}x^{i-j}h^{j}\right)}\right]$

$=\prod_{i=1}^{n}\left[\prod_{j=1}^{i}\left(a_{i}^{\left({i \choose j}h^{j}\right)}\right)^{\left(x^{i-j}\right)}\right]$

$=\prod_{i=1}^{n}\left[\prod_{j=1}^{i}b_{i,j}^{\left(x^{i-j}\right)}\right]$,
where $b_{i,j}=a_{i}^{\left({i \choose j}h^{j}\right)}$

$=\prod_{i=1}^{n}\left[b_{i,1}^{\left(x^{i-1}\right)}\varoast_{l,k}b_{i,2}^{\left(x^{i-2}\right)}\varoast_{l,k}\cdots\varoast_{l,k}b_{i,i}\right]$

$=b_{1,1}\varoast_{l,k}\left(b_{2,1}^{\left(x\right)}\varoast_{l,k}b_{2,2}\right)\varoast_{l,k}\cdots\varoast_{l,k}\left(b_{n,1}^{\left(x^{n-1}\right)}\varoast_{l,k}b_{n,2}^{\left(x^{n-2}\right)}\varoast_{l,k}\cdots\varoast_{l,k}b_{n,n}\right)$

$=b_{n,1}^{\left(x^{n-1}\right)}\varoast_{l,k}\left(b_{n,2}\varoast_{l,k}b_{n-1,1}\right)^{\left(x^{n-2}\right)}\varoast_{l,k}\cdots\varoast_{l,k}\left(\prod_{i=1}^{n}b_{i,i}\right)$.

So, $x\mapsto f\left(x+h\right)f\left(x\right)^{-1}$ is a polynomial
of degree $n-1$. Hence, $f:\left(\mathbb{Z},+\right)\rightarrow\left(\mathbb{Z},\varoast_{l,k}\right)$
is a polynomial of degree $n$. 

\noindent Let us take $\mathbb{P}\left(\left(\mathbb{Z},+\right),\left(\mathbb{Z},\varoast_{l,k}\right)\right)$
be the set of all polynomials $f:\left(\mathbb{Z},+\right)\rightarrow\left(\mathbb{Z},\varoast_{l,k}\right)$
with $f\left(0\right)=-\frac{k-1}{l}$. ( Since, $-\frac{k-1}{l}$
is the identity element of $\left(\mathbb{Z},\varoast_{l,k}\right)$.
)

Now, taking $\Gamma=\mathbb{P}\left(\left(\mathbb{Z},+\right),\left(\mathbb{Z},\varoast_{l,k}\right)\right)$
and by \cite[Example 4.7]{key-1}, $\Gamma$ is a licit ( Definition
\ref{licit} (2)).

Note that $\Gamma$ is a $R$-family ( Definition \ref{R-f} (1)),
directly follows from \cite[Theorem 2.11]{key-1}. Now, we have the
following analogue of Polynomial Central Sets Theorem.
\begin{thm}
\label{C.S.T.}Let $\left\langle y_{\alpha}\right\rangle _{\alpha\in\mathcal{F}}$
be an $IP$-set in $\left(\mathbb{Z},+\right)$. Let $F\subseteq\mathbb{P}\left(\left(\mathbb{Z},+\right),\left(\mathbb{Z},\varoast_{l,k}\right)\right)$
and $A\subseteq\left(\mathbb{Z},\varoast_{l,k}\right)$ is a central
set. Then there exists a sequence $\left\langle x_{n}\right\rangle _{n=1}^{\infty}$
in $\left(\mathbb{Z},\varoast_{l,k}\right)$ and a sub-IP set $\left\langle z_{\beta}\right\rangle _{\beta\in\mathcal{F}}$
of $\left\langle y_{\alpha}\right\rangle _{\alpha\in\mathcal{F}}$
such that for all $f\in F$ and for all $\beta\in\mathcal{F}$ we
have 
\[
\mathfrak{G}_{l,k}\left(x_{\beta}^{<},f\left(z_{\beta}\right)\right)\in A.
\]
\end{thm}

\begin{proof}
From \cite[Theorem 3.8]{key-1}, there exists such sequences and 

$x_{\beta}\varoast_{l,k}f\left(z_{\beta}\right)=\prod_{i\in\beta}x_{i}\varoast_{l,k}f\left(z_{\beta}\right)=\mathfrak{G}_{l,k}\left(x_{\beta}^{<},f\left(z_{\beta}\right)\right)\in A$.
\end{proof}
Now, if $F\subseteq\mathbb{P}\left(\left(\mathbb{Z},+\right),\left(\mathbb{Z},\varoast_{l,k}\right)\right)$
be a finite set of polynomials with coefficients 
\[
\left\{ a_{1}^{\left(i\right)},a_{2}^{\left(i\right)},\ldots,a_{d}^{\left(i\right)}\right\} _{i=1}^{m},
\]
 where $\left|F\right|=m$, then by the above theorem we get the following
corollary,
\begin{cor}
Let $\left\{ a_{1}^{\left(i\right)},a_{2}^{\left(i\right)},\ldots,a_{d}^{\left(i\right)}\right\} _{i=1}^{m}\subseteq\mathbb{Z}\setminus\left\{ -\frac{k}{l},-\frac{k+1}{l}\right\} $
be a finite set. Then for any finite partition of $\mathbb{Z}$, and
any $IP$-set $\text{FS}\left(\left\langle y_{n}\right\rangle _{n=1}^{\infty}\right)$,
there exist a sequence $\left\langle x_{n}\right\rangle _{n=1}^{\infty}$
and a sub-$IP$-set $\text{FS}\left(\left\langle z_{n}\right\rangle _{n=1}^{\infty}\right)$
such that 
\begin{align*}
\left\{ \frac{1}{l}\left[\prod_{j\in\beta}\left(lx_{j}+k\right)\left(la_{1}^{\left(i\right)}+k\right)^{z_{\beta}}\left(la_{2}^{\left(i\right)}+k\right)^{z_{\beta}^{2}}\cdots\left(la_{d}^{\left(i\right)}+k\right)^{z_{\beta}^{d}}-k\right]\right.\\
\mid\beta\in\mathcal{P}_{f}\left(\mathbb{N}\right)\Biggr\}  ,
\end{align*}
 is monochromatic.
\end{cor}

The following example shows one of the most simple monochromatic configuration
one can get from Theorem \ref{C.S.T.}.
\begin{example}
Let $d,k,n\in\mathbb{N}$, $\left(l,k\right)=\left(2,1\right)$ and
take the finite sequence
\begin{align*}
A=\left(0,0,\ldots,0\right)\cup\left\{ \left(i,0,\ldots,0\right),\left(0,i,\ldots,0\right),\right.\\
\left.\ldots,\left(0,0,\ldots,i\right)\,|1\leq i\leq n.\right\} .
\end{align*}
Then, there exists a sub-$IP$-set $\text{FS}\left(\left\langle z_{n}\right\rangle _{n=1}^{\infty}\right)$
such that 
\begin{align*}
\left\lbrace \frac{1}{2}\left[\prod_{j\in\beta}\left(2x_{j}+1\right)\left(2a_{1}^{\left(i\right)}+1\right)^{z_{\beta}}\left(2a_{2}^{\left(i\right)}+1\right)^{z_{\beta}^{2}}\cdots\left(2a_{d}^{\left(i\right)}+1\right)^{z_{\beta}^{d}}-1\right]\right.\\
\Big\vert \beta\in\mathcal{P}_{f}\left(\mathbb{N}\right) \Biggr\} 
\end{align*}

\noindent where $\left\{ a_{1}^{\left(i\right)},a_{2}^{\left(i\right)},\ldots,a_{d}^{\left(i\right)}\right\} \subseteq A$,
is monochromatic. 

Let, for $1\leq j\leq n$, $a_{j}=2x_{j}+1$, then the simple form
of the monochromatic form is as follows:
\pagebreak
\[
\frac{1}{2}\left(a_{1}-1\right),\frac{1}{2}\left(a_{1}3^{z_{1}}-1\right),\frac{1}{2}\left(a_{1}5^{z_{1}}-k\right),\ldots,\frac{1}{2}\left(a_{1}\left(2n+1\right)^{z_{1}}-1\right),
\]

\[
\frac{1}{2}\left(a_{1}3^{z_{1}^{2}}-1\right),\frac{1}{2}\left(a_{1}5^{z_{1}^{2}}-1\right),\ldots,\frac{1}{2}\left(a_{1}\left(2n+1\right)^{z_{1}^{2}}-1\right),
\]

\[
\vdots
\]

\[
\frac{1}{2}\left(a_{1}3^{z_{1}^{d}}-1\right),\frac{1}{2}\left(a_{1}5^{z_{1}^{d}}-1\right),\ldots,\frac{1}{2}\left(a\left(2n+1\right)^{z_{1}^{d}}-1\right),
\]
\vspace{.2in}
\[
\frac{1}{2}\left(a_{2}-1\right),\frac{1}{2}\left(a_{2}3^{z_{2}}-1\right),\frac{1}{2}\left(a_{2}5^{z_{2}}-1\right),\ldots,\frac{1}{2}\left(a_{2}\left(2n+1\right)^{z_{2}}-1\right),
\]

\[
\frac{1}{2}\left(a_{2}3^{z_{2}^{2}}-1\right),\frac{1}{2}\left(a_{2}5^{z_{2}^{2}}-1\right),\ldots,\frac{1}{2}\left(a_{2}\left(2n+1\right)^{z_{2}^{2}}-1\right),
\]

\[
\vdots
\]

\[
\frac{1}{2}\left(a_{2}3^{z_{2}^{d}}-1\right),\frac{1}{2}\left(a_{2}5^{z_{2}^{d}}-1\right),\ldots,\frac{1}{2}\left(a_{2}\left(2n+1\right)^{z_{2}^{d}}-1\right),
\]
\vspace{.2in}
{\small{}
\[
\frac{1}{2}\left(a_{1}a_{2}-1\right),\frac{1}{2}\left(a_{1}a_{2}3^{z_{1}+z_{2}}-1\right),\frac{1}{2}\left(a_{1}a_{2}5^{z_{1}+z_{2}}-1\right),\ldots,\frac{1}{2}\left(a_{1}a_{2}\left(2n+1\right)^{z_{1}+z_{2}}-1\right),
\]
}{\small\par}

{\small{}
\[
\frac{1}{2}\left(a_{1}a_{2}3^{\left(z_{1}+z_{2}\right)^{2}}-1\right),\frac{1}{2}\left(a_{1}a_{2}5^{\left(z_{1}+z_{2}\right)^{2}}-1\right),\ldots,\frac{1}{2}\left(a_{1}a_{2}\left(2n+1\right)^{\left(z_{1}+z_{2}\right)^{2}}-1\right),
\]
}{\small\par}

{\small{}
\[
\vdots
\]
}{\small\par}

{\small{}
\[
\frac{1}{2}\left(a_{1}a_{2}3^{\left(z_{1}+z_{2}\right)^{d}}-1\right),\frac{1}{2}\left(a_{1}a_{2}5^{\left(z_{1}+z_{2}\right)^{d}}-1\right),\ldots,\frac{1}{2}\left(a_{1}a_{2}\left(2n+1\right)^{\left(z_{1}+z_{2}\right)^{d}}-1\right)
\]
}{\small\par}

{\small{}
\[
\noindent \vdots \hspace{.5in} \vdots \hspace{.5in} \vdots
\]
}{\small\par}
\end{example}

\section{A variant of the Polynomial Deuber's Theorem}

Continuing from the previous section, let us now introduce multi-variable
polynomials from $\left(\mathbb{Z}^{m},+\right)$ to $\left(\mathbb{Z},\varoast_{l,k}\right)$
for any $m$. First note that, a multi-variable polynomial of degree
$n$, $P:\left(\mathbb{Z}^{m},+\right)\rightarrow\left(\mathbb{Z},+\right)$
is of the form 
\[
P\left(x_{1},x_{2},\ldots,x_{m}\right)=\sum_{i_{1}+i_{2}+\cdots+i_{m}\leq n}\alpha_{i_{1}i_{2}\cdots i_{m}}x_{1}^{i_{1}}x_{2}^{i_{2}}\cdots x_{m}^{i_{m}}.
\]
 So, as from the discussion of the previous section, a polynomial
of $m$ variable and of degree $n$ from $\left(\mathbb{Z}^{m},+\right)$
to $\left(\mathbb{Z},\varoast_{l,k}\right)$ should be of the form,
\[
P\left(x_{1},x_{2},\ldots,x_{m}\right)=\sum_{i_{1}+i_{2}+\cdots+i_{m}\leq n}\alpha_{i_{1}i_{2}\cdots i_{m}}^{\left(x_{1}^{i_{1}}x_{2}^{i_{2}}\cdots x_{m}^{i_{m}}\right)}.
\]
 It is straightforward to verify that this is an $n$ degree polynomial.

Now, as it is clear that $\left(\mathbb{Z},+\right)$ and $\left(\mathbb{Z},\varoast_{l,k}\right)$
are two different groups, so we cannot apply \cite[Theorem 4.9]{key-1},
but a weak version of that theorem holds in our care.

The following definition is a variant of the so called $\left(m,p,c\right)$-set.
\begin{defn}
Let $m\in\mathbb{N}$ and $\vec{F}$ is an $m$-tuple $\vec{F}=\left\{ F_{1},F_{2},\ldots,F_{m}\right\} $,
where $F_{i}\subseteq\mathbb{P}\left(\left(\mathbb{Z}^{i},+\right),\left(\mathbb{Z},\varoast_{l,k}\right)\right)$.
Then, for any given $IP$-set $\left\langle S_{\alpha}\right\rangle _{\alpha\in\mathcal{P}_{f}\left(\mathbb{N}\right)}\subseteq\left(\mathbb{Z}\setminus\left\{ 0\right\} \right)^{m+1}$
and $B$, the collection of $m+1$ base sequences $B=\left\{ B_{0},B_{1},\ldots,B_{m}\right\} $,
the set $D\left(m,\vec{F},\varoast_{l,k},S_{\alpha}^{<},B\right)$
is defined by,
\[
D\left(m,\vec{F},\varoast_{l,k},S_{\alpha}^{<},B\right)=\left\{ \begin{array}{cc}
\mathfrak{G}_{l,k}\left(S_{\alpha,0}^{<\left(B_{0}\right)}\right)\\
\mathfrak{G}_{l,k}\left(S_{\alpha,1}^{<\left(B_{1}\right)},f\left(S_{\alpha,0}\right)\right) & f\in F_{1}\\
\mathfrak{G}_{l,k}\left(S_{\alpha,2}^{<\left(B_{2}\right)},f\left(S_{\alpha,0},S_{\alpha,1}\right)\right) & f\in F_{2}\\
\vdots & \vdots\\
\mathfrak{G}_{l,k}\left(S_{\alpha,m}^{<\left(B_{m}\right)},f\left(S_{\alpha,0},\ldots,S_{\alpha,m-1}\right)\right) & f\in F_{m}
\end{array}\right\} .
\]
\end{defn}

\begin{thm}
\label{P.deu} Let $l,k,m\in\mathbb{Z}$, $\mathbb{Z}=C_{1}\cup C_{2}\cup\cdots\cup C_{r}$
be any finite partition of $\mathbb{Z}$ and for $1\leq i\leq m$,
$F_{i}\subseteq\mathbb{P}\left(\left(\mathbb{Z}^{i},+\right),\left(\mathbb{Z},\varoast_{l,k}\right)\right)$
be a finite collection of polynomials. Then there exists an $IP$-set
$\left\langle S_{\alpha}\right\rangle _{\alpha\in\mathcal{P}_{f}\left(\mathbb{N}\right)}$
in $\left(\mathbb{Z}^{m+1},+\right)$ and a collection of $m+1$ base
sequences $B=\left\{ B_{0},B_{1},\ldots,B_{m}\right\} $ such that
for some $i\in\left\{ 1,2,\ldots,r\right\} $, 
\[
D\left(m,\vec{F},\varoast_{l,k},S_{\alpha}^{<},B\right)\in C_{i}.
\]
\end{thm}

\begin{proof}
Let for $i\in\left\{ 1,2,\ldots,r\right\} $, $C_{i}$ is central
in $\left(\mathbb{Z},\varoast_{l,k}\right)$. So, there exists an
injective sequence $\left\langle x_{n,0}\right\rangle _{n=1}^{\infty}$
such that $\mathfrak{G}_{l,k}\left(x_{n,0}\right)_{n=1}^{\infty}\in C_{i}$.
Passing to subsequence if necessary, we can assume the finite sums
are distinct, i.e; $x_{\alpha,0}\neq x_{\beta,0}$ if $\alpha\neq\beta$. 

Let, $B_{0}=\left\{ x_{n,0}:n\in\mathbb{N}\right\} $. Then, $B_{0}\subseteq C_{i}$.
Then applying \cite[Theorem 4.10]{key-1}, we obtain a sub-$IP$-set
$\text{FS}\left(\left\langle y_{n,0}\right\rangle _{n=1}^{\infty}\right)$
of $\text{FS}\left(\left\langle x_{n,0}\right\rangle _{n=1}^{\infty}\right)$
and an injective sequence $B_{1}=\left\{ x_{n,1}:n\in\mathbb{N}\right\} $
such that $\mathfrak{G}_{l,k}\left(x_{\beta,1}^{<\left(B_{1}\right)},f\left(y_{\beta,0}\right)\right)\in C_{i}$
for all $f\in F_{1}$ and $\beta\in\mathcal{P}_{f}\left(\mathbb{N}\right)$.

Again, passing to a subsequence if necessary, we have $x_{\alpha,1}\neq x_{\beta,1}$
if $\alpha\neq\beta$. So, we have an $IP$-set $S_{\alpha}=\left(S_{\alpha,0},S_{\alpha,1}\right)$
and two base sequences $B_{0},B_{1}$ such that 
\[
\mathfrak{G}_{l,k}\left(S_{\alpha,0}^{<\left(B_{0}\right)}\right)\in C_{i}
\]
 and 
\[
\mathfrak{G}_{l,k}\left(S_{\alpha,1}^{<\left(B_{1}\right)},f\left(S_{\alpha,0}\right)\right)\in C_{i}
\]
for all $f\in F_{1}$ and all $\alpha\in\mathcal{P}_{f}\left(\mathbb{N}\right)$.
Here we have taken $S_{n}=\left(S_{n,0},S_{n,1}\right)=\left(y_{n,0},x_{n,1}\right)$.

So, by \cite[Theorem 4.10]{key-1}, we obtain a sub-$IP$-set $\text{FS}\left(\left\langle S'_{n}\right\rangle _{n=1}^{\infty}\right)$
of $S_{\alpha}$ and an injective sequence $B_{2}=\left\{ x_{n,2}:n\in\mathbb{N}\right\} $
such that $\mathfrak{G}_{l,k}\left(x_{\beta,2}^{<\left(B_{2}\right)},f\left(S'_{\beta}\right)\right)\in C_{i}$
for all $f\in F_{2}$ and $\beta\in\mathcal{P}_{f}\left(\mathbb{N}\right)$.

Proceeding in a similar way, passing to a subsequence if necessary
we may assume $x_{\alpha,2}\neq x_{\beta,2}$ if $\alpha\neq\beta$.
So, we have an $IP$-set $S_{\alpha}=\left(S_{\alpha,0},S_{\alpha,1},S_{\alpha,2}\right)$
and three base sequences $B_{0},B_{1},B_{2}$ such that 
\[
\mathfrak{G}_{l,k}\left(S_{\alpha,0}^{<\left(B_{0}\right)}\right)\in C_{i},
\]
\[
\mathfrak{G}_{l,k}\left(S_{\alpha,1}^{<\left(B_{1}\right)},f\left(S_{\alpha,0}\right)\right)\in C_{i},
\]
 and
\[
\mathfrak{G}_{l,k}\left(S_{\alpha,2}^{<\left(B_{2}\right)},f\left(S_{\alpha,0},S_{\alpha,1}\right)\right)\in C_{i}
\]
for all $f\in F_{2}$ and all $\alpha\in\mathcal{P}_{f}\left(\mathbb{N}\right)$. 

Here, $S_{\alpha}=\left(S_{\alpha,0},S_{\alpha,1},S_{\alpha,2}\right)=\left(S'_{\alpha,0},S'_{\alpha,1},b_{\alpha,2}\right)$.

Iterating this argument we have the desired result.
\end{proof}
\begin{rem}
In Theorem \ref{P.deu}, if we choose identity polynomial map in each
$F_{i}$, then we have $\mathfrak{G}_{l,k}\left(S_{\alpha,j}^{<\left(B_{j}\right)}\right)\in C_{i}$
for all $j\in\left\{ 1,2,\ldots,m\right\} $.
\end{rem}

\begin{cor}
\label{mono}Let $n\in\mathbb{N}$ and $\left(l,k\right)=\left(1,0\right)$.
Take $\left\langle a_{j}\right\rangle _{j=1}^{n}\subseteq\mathbb{N}$.
Then, for any finite partition of $\mathbb{Z}$, there exist $x,b_{1},b_{2},\ldots,b_{k}$
such that 
\[
\left\{ x,b_{1},b_{2},\ldots,b_{k},\prod_{i=1}^{k}b_{i},\left\{ x\cdot\left(a_{j}^{\sum_{i=1}^{k}b_{i}}\right)\right\} _{j=1}^{n}\right\} 
\]
 is monochromatic.
\end{cor}

\begin{proof}
Let $F_{1}=\left\{ a_{j}^{\left(x\right)}:1\leq i\leq n\right\} $.
Then, there exist base sequences $B_{0},B_{1}$ and an $IP$-set $\left\langle S_{\alpha}\right\rangle _{\alpha\in\mathcal{P}_{f}\left(\mathbb{N}\right)}$
in $\left(\mathbb{Z}^{2},+\right)$ with the above property. 

Let $S_{\left\{ 1\right\}}=\left(S_{\left\{ 1\right\} ,0},S_{\left\{ 1\right\} ,1}\right)$.
Then, $S_{\left\{ 1\right\} ,0}=\sum_{i=1}^{k}b_{i}$ and $S_{\left\{ 1\right\} ,1}=\sum_{s=1}^{m}c_{s}$,
where 
\[
\left\{ b_{1},b_{2},\ldots,b_{k}\right\} \subseteq B_{0},\,\left\{ c_{1},c_{2},\ldots,c_{m}\right\} \subseteq B_{1}.
\]

Now, 
\[
\mathfrak{G}_{1,0}\left(S_{\left\{ 1\right\} ,0}^{<\left(B_{0}\right)}\right)=\mathfrak{G}_{1,0}\left(b_{1},b_{2},\ldots,b_{k}\right)=b_{1}b_{2}\cdots b_{k}.
\]
 Similarly, 
\[
\mathfrak{G}_{1,0}\left(S_{\left\{ 1\right\} ,1}^{<\left(B_{1}\right)}\right)=\mathfrak{G}_{1,0}\left(c_{1},c_{2},\ldots,c_{s}\right)=c_{1}c_{2}\cdots c_{s}.
\]

So, 

$\mathfrak{G}_{1,0}\left(S_{\left\{ 1\right\} ,1}^{<\left(B_{1}\right)},a_{j}^{\left(S_{1,0}\right)}\right)$

$=\mathfrak{G}_{1,0}\left(c_{1},c_{2},\ldots,c_{s},a_{j}^{\left(b_{1}+b_{2}+\cdots+b_{k}\right)}\right)$

$=\mathfrak{G}_{1,0}\left(c_{1},c_{2},\ldots,c_{s}\right)\varoast_{1,0}\mathfrak{G}_{1,0}\left(a_{j}^{\left(b_{1}+b_{2}+\cdots+b_{k}\right)}\right)$

$=c_{1}c_{2}\cdots c_{s}\cdot a_{j}^{\sum_{i=1}^{k}b_{i}}$ for all
$j\in\left\{ 1,\ldots,n\right\} $.

Let, $x=c_{1}c_{2}\cdots c_{s}$. Then the following configuration
\[
\left\{ x,b_{1},b_{2},\ldots,b_{k},\prod_{i=1}^{k}b_{i},\left\{ x\cdot\left(a_{j}^{\sum_{i=1}^{k}b_{i}}\right)\right\} _{j=1}^{n}\right\} 
\]
 is monochromatic.
\end{proof}
\begin{example}
For $1\leq j\leq n$, choose $a_{j}=j$. Then, by Corollary \ref{mono},
the following configuration 
\[
\left\{ x,b_{1},b_{2},\ldots,b_{k},\prod_{i=1}^{k}b_{i},\left\{ x\cdot\left(j^{\sum_{i=1}^{k}b_{i}}\right)\right\} _{j=1}^{n}\right\} 
\]
 is monochromatic.

i.e; 
\[
\left\{ x,b_{1},b_{2},\ldots,b_{k},\prod_{i=1}^{k}b_{i},x\cdot2^{\sum_{i=1}^{k}b_{i}},\ldots,x\cdot n^{\sum_{i=1}^{k}b_{i}}\right\} 
\]
 is monochromatic.
\end{example}

\begin{example}
For any finite partition of $\mathbb{Z}=\bigcup_{i=1}^{r}C_{i}$,
there exist $m,n,p\in\mathbb{N}$ and finite sequences $\left\langle a_{i}\right\rangle _{i=1}^{m}$,
$\left\langle b_{j}\right\rangle _{j=1}^{n}$ and $\left\langle c_{q}\right\rangle _{q=1}^{p}$
such that the following configuration 
\begin{align*}
\left\{ \prod_{i=1}^{m}a_{i},\prod_{j=1}^{n}b_{j},\prod_{q=1}^{p}c_{q},\prod_{j=1}^{n}b_{j}\cdot2^{\sum_{i=1}^{m}a_{i}},\prod_{j=1}^{n}b_{j}\cdot3^{\sum_{i=1}^{m}a_{i}},\right.\\
\left.\prod_{q=1}^{p}c_{q}\cdot2^{\left(\sum_{i=1}^{m}a_{i}\right)\left(\sum_{j=1}^{n}b_{j}\right)},\prod_{q=1}^{p}c_{q}\cdot3^{\left(\sum_{i=1}^{m}a_{i}\right)\left(\sum_{j=1}^{n}b_{j}\right)}\right\} 
\end{align*}
 is monochromatic.
\end{example}

\begin{proof}
Let $\left(l,k\right)=\left(1,0\right)$. Consider the following finite
family of functions $\left\{ F_{1},F_{2}\right\} $, where 
\[
F_{1}=\left\{ 1^{\left(x\right)},2^{\left(x\right)},3^{\left(x\right)}\right\} ,F_{2}=\left\{ 1^{\left(xy\right)},2^{\left(xy\right)},3^{\left(xy\right)}\right\} .
\]
Now use Theorem \ref{P.deu} for the above set of functions. So, there
exist three base sequences $B_{0},B_{1}$ and $B_{2}$ from Theorem
\ref{P.deu} and choose $\alpha=\left\{ 1\right\} $. Let, $m_{1},n_{1}\in\mathbb{N}$
and by letting $S_{\left\{ 1\right\} ,0}=\sum_{m_{1}}^{n_{1}}x_{i}$,
where $\left\{ x_{i}\right\} _{i=m_{1}}^{n_{1}}\subset B_{0}$, we
have

$\mathfrak{G}_{1,0}\left(S_{\left\{ 1\right\} ,0}^{<\left(B_{0}\right)}\right)$

$=\mathfrak{G}_{1,0}\left(x_{m_{1}},\ldots,x_{n_{1}}\right)$

$=x_{m_{1}}\cdots x_{n_{1}}$.

Again, let $m_{2},n_{2}\in\mathbb{N}$ and by letting $S_{\left\{ 1\right\} ,1}=\sum_{m_{2}}^{n_{2}}y_{j}$,
where $\left\{ y_{j}\right\} _{j=m_{2}}^{n_{2}}\subset B_{1}$, we
have

$\mathfrak{G}_{1,0}\left(S_{\left\{ 1\right\} ,1}^{<\left(B_{1}\right)},f\left(S_{\left\{ 1\right\} ,0}\right)\right)$,
where $f\in F_{1}$

$=\mathfrak{G}_{1,0}\left(y_{m_{2}},\ldots,y_{n_{2}},f\left(\sum_{i=m_{1}}^{n_{1}}x_{i}\right)\right)$

$=\left\{ \begin{array}{cc}
y_{m_{2}}\cdots y_{n_{2}} & \text{if, }f=1^{\left(x\right)}\\
y_{m_{2}}\cdots y_{n_{2}}\cdot2^{\sum_{m_{1}}^{n_{1}}x_{i}} & \text{if, }f=2^{\left(x\right)}\\
y_{m_{2}}\cdots y_{n_{2}}\cdot3^{\sum_{m_{1}}^{n_{1}}x_{i}} & \text{if, }f=3^{\left(x\right)}
\end{array}\right.$.

Also, let $m_{3},n_{3}\in\mathbb{N}$ and by letting $S_{\left\{ 1\right\} ,2}=\sum_{m_{3}}^{n_{3}}z_{3}$,
where $\left\{ z_{k}\right\} _{k=m_{3}}^{n_{3}}\subset B_{2}$, we
have

$\mathfrak{G}_{1,0}\left(S_{\left\{ 1\right\} ,2}^{<\left(B_{2}\right)},f\left(S_{\left\{ 1\right\} ,0},S_{\left\{ 1\right\} ,1}\right)\right)$,
where $f\in F_{2}$

$=\mathfrak{G}_{1,0}\left(z_{m_{3}},\ldots,z_{n_{3}},f\left(\sum_{i=m_{1}}^{n_{1}}x_{i},\sum_{m_{2}}^{n_{2}}y_{j}\right)\right)$

$=\left\{ \begin{array}{cc}
z_{m_{3}}\cdots z_{n_{3}} & \text{if, }f=1^{\left(xy\right)}\\
z_{m_{3}}\cdots z_{n_{3}}\cdot2^{\left(\sum_{i=m_{1}}^{n_{1}}x_{i}\right)\left(\sum_{m_{2}}^{n_{2}}y_{j}\right)} & \text{if, }f=2^{\left(xy\right)}\\
z_{m_{3}}\cdots z_{n_{3}}\cdot3^{\left(\sum_{i=m_{1}}^{n_{1}}x_{i}\right)\left(\sum_{m_{2}}^{n_{2}}y_{j}\right)} & \text{if, }f=3^{\left(xy\right)}
\end{array}\right.$.

Now replace $\left\{ x_{i}\right\} _{i=m_{1}}^{n_{1}}$ by $\left\{ a_{i}\right\} _{i=1}^{m1}$,
$\left\{ y_{j}\right\} _{j=m_{2}}^{n_{2}}$ by $\left\{ b_{j}\right\} _{j=1}^{n}$
and $\left\{ z_{k}\right\} _{k=m_{3}}^{n_{3}}$by $\left\{ c_{q}\right\} _{q=1}^{p}$
to obtain the desired configuration.
\end{proof}

\section{A New Approach to Additive and Multiplicative Operation}

We will now introduce and study two binary operations, additive and
multiplicative and their iterative versions, which are associative.
Also, we will provide some examples. These operations are much complicated
to handle and in this article we will only use the Hales-Jewett theorem.

\subsection{ A new additive operation}

Let $n\in\mathbb{N}$ and let us define the function $f:\mathbb{N}\rightarrow\omega$,
by $f\left(n\right)=\max\left\{ x:2^{x}\vert n\right\} $. As, each
$n\in\mathbb{N}$ can be written as $n=2^{x_{n}}\left(2y_{n}-1\right)$
in a unique way, where $x_{n}=f\left(n\right)$, itt is easy to verify
that the function $\varphi:\mathbb{N}\rightarrow\omega\times\mathbb{N}$,
defined by 
\[
\varphi\left(n\right)=\left(x_{n},y_{n}\right)=\left(f\left(n\right),\frac{1}{2}\left(\frac{n}{2^{f\left(n\right)}}+1\right)\right),
\]
is a bijection. 

Take the commutative semigroup $\left(\omega\times\mathbb{N},+\right)$,
where the operation $+$ is defined as $\left(a,b\right)+\left(c,d\right)=\left(a+b,c+d\right)$.
Then, the bijection $\varphi$ induces an associative operation $\oplus$
on $\mathbb{N}$, defined by, 
\[
p=m\oplus n
\]
\[
\text{if and only if},\ \varphi\left(p\right)=\varphi\left(m\right)+\varphi\left(n\right)
\]

\[
\text{if and only if},\ \varphi\left(p\right)=\left(x_{m}+x_{n},y_{m}+y_{n}\right)
\]

\[
\text{if and only if},\ p=2^{x_{m}+x_{n}}\left(2\left(y_{m}+y_{n}\right)-1\right)
\]

\[
\text{if and only if},\ p=2^{f\left(m\right)+f\left(n\right)}\left(\frac{m}{2^{f\left(m\right)}}+1+\frac{n}{2^{f\left(n\right)}}+1-1\right)
\]

\[
\text{if and only if},\ p=2^{f\left(m\right)+f\left(n\right)}\left(\frac{m}{2^{f\left(m\right)}}+\frac{n}{2^{f\left(n\right)}}+1\right).
\]
So, we have $m\oplus n=2^{f\left(m\right)+f\left(n\right)}\left(\frac{m}{2^{f\left(m\right)}}+\frac{n}{2^{f\left(n\right)}}+1\right)$. 

It can be easily seen that 
\[
a_{1}\oplus a_{2}\oplus\cdots\oplus a_{n}=2^{\sum_{i=1}^{n}f\left(a_{i}\right)}\left(\sum_{i=1}^{n}\frac{a_{i}}{2^{f\left(a_{i}\right)}}+\left(n-1\right)\right).
\]

We will now use the Hales-Jewett theorem to deduce some new Ramsey
theoretic configurations.
\begin{thm}
Let $r\in\mathbb{N}$ and $a_{1},a_{2},\ldots,a_{n}$ are distinct
natural numbers. Then for every $r$-partition of $\mathbb{N}$, there
exist $x,y$ and $c$ in $\mathbb{N}$ such that 
\[
\left\{ x2^{cf\left(a\right)}\left(y+c\frac{a}{2^{f\left(a\right)}}\right):a\in\left\{ a_{1},a_{2},\ldots,a_{n}\right\} \right\} 
\]
 is monochromatic.
\end{thm}

\begin{proof}
Let $\mathbb{A}=\left\{ a_{1},a_{2},\ldots,a_{n}\right\} $ and $r\in\mathbb{N}$.
Then choose the Hales-Jewett number $N=N\left(\mathbb{A},r\right)$.

Now consider the word space $\mathbb{A}^{N}$ and take the correspondence
map $g:\mathbb{A}^{N}\rightarrow\mathbb{N}$ defined by,

$g\left(a_{1},a_{2},\ldots,a_{n}\right)$

$=a_{1}\oplus a_{2}\oplus\cdots\oplus a_{n}$

$=2^{\sum_{i=1}^{n}f\left(a_{i}\right)}\left(\sum_{i=1}^{n}\frac{a_{i}}{2^{f\left(a_{i}\right)}}+\left(n-1\right)\right)$. 

Now every $r$-partition on $\mathbb{N}$ induces a $r$-partition
on $\mathbb{A}^{N}$. Then from Hales-Jewett theorem and above configuration,
there exist $c,x,y\in\mathbb{N}$ such that 
\[
\left\{ x2^{cf\left(a\right)}\left(y+c\frac{a}{2^{f\left(a\right)}}\right):a\in\mathbb{A}\right\} 
\]
 is monochromatic.
\end{proof}
The following three facts can be deduced from Hales-Jewett theorem
directly, but here we provide these as an application of our theorem.
\begin{fact}
Let $r\in\mathbb{N}$. For any $l\in\mathbb{N}$, choose $F=\left\{ 2,2^{2},\ldots,2^{l}\right\} $.
Denote $a_{i}=2^{i}$, for $i\in\left\{ 1,2,\ldots,l\right\} $. Then
$f\left(a_{i}\right)=i$ and $\frac{a_{i}}{2^{f\left(a_{i}\right)}}=1$.
Then from the above theorem, there exist $x,y,c\in\mathbb{N}$ such
that
\[
x2^{cf\left(a_{i}\right)}\left(y+c\frac{a_{i}}{2^{f\left(a_{i}\right)}}\right)=x2^{ci}\left(y+c\right)=ab^{i},
\]
 for all $i\in\left\{ 1,2,\ldots,l\right\} $, where $a=x\left(y+c\right),b=2^{c}$.
So, we have the geometric progressions.
\end{fact}

\begin{fact}
Let $n,r\in\mathbb{N}$ and $p_{1},p_{2},\ldots,p_{n}$ be different
odd primes. Then for any $r$-partition of $\mathbb{N}$, from the
above theorem there exist $c,x,y\in$ $\mathbb{N}$ such that the
following configuration
\[
\left\{ xy+xcp_{1},xy+xcp_{2},\ldots,xy+xcp_{n}\right\} ,
\]
is monochromatic as $f\left(p_{i}\right)=0$ and $\frac{p_{i}}{2^{f\left(p_{i}\right)}}=p_{i}$.
\end{fact}

\begin{fact}
Let $n,r\in\mathbb{N}$ and $p$ be an odd prime. Let $p,p^{2},\ldots,p^{n}$.
Then $f\left(p^{i}\right)=0$ and $\frac{p^{i}}{2^{f\left(p^{i}\right)}}=p^{i}$
for all $i\in\left\{ 1,2,\ldots,n\right\} $. Then for any $r$-partition
of $\mathbb{N}$, from the above theorem there exist $c,x,y\in$ $\mathbb{N}$
such that the following configuration
\[
\left\{ xy+xcp,xy+xcp^{2},\ldots,xy+xcp^{n}\right\} 
\]
is monochromatic.
\end{fact}

The following examples are explicitly showing our results in a simple
way.
\begin{example}
Let us consider the numbers $pq,2p,2q$, where $p,q$ are odd primes.
Now $f\left(pq\right)=0,f\left(2p\right)=f\left(2q\right)=2$ and
$\frac{pq}{2^{f\left(pq\right)}}=pq,$$\frac{p}{2^{f\left(p\right)}}=p$
and $\frac{q}{2^{f\left(q\right)}}=q$.

Then for any $r$-partition of $\mathbb{N}$, from the above theorem
there exist $c,x,y\in$ $\mathbb{N}$ such that the following monochromatic
configuration:
\[
\left\{ x\left(y+cpq\right),2^{c}x\left(y+cp\right),2^{c}x\left(y+cq\right)\right\} .
\]
\end{example}

\begin{example}
Let us consider the numbers $2^{3}p,p^{3}$, where $p$ is an odd
prime. Now $f\left(p^{3}\right)=0,f\left(2^{3}p\right)=3$ and $\frac{p^{3}}{2^{f\left(p^{3}\right)}}=p^{3},$$\frac{2^{3}p}{2^{f\left(2^{3}p\right)}}=p$.
Then for any $r$-partition of $\mathbb{N}$, from the above theorem
there exist $c,x,y\in$ $\mathbb{N}$ such that the following monochromatic
configuration:
\[
\left\{ x\left(y+cp^{3}\right),x2^{3c}\left(y+cp\right)\right\} .
\]
Taking $a=xy,b=xc$ we can say for any $r$-partition of $\mathbb{N}$,
from the above theorem there exist $c,a,b\in$ $\mathbb{N}$ such
that the following monochromatic configuration:
\[
\left\{ a8^{c}+bp,a+bp^{3}\right\} .
\]
\end{example}

\begin{example}
Similarly considering the numbers $2p^{2},4p^{4}$, where $p$ is
a prime, we see that for any $r$-partition of $\mathbb{N}$, from
the above theorem there exist $c,a,b\in$ $\mathbb{N}$ such that
the following monochromatic configuration:
\[
\left\{ 2^{c}\left(a+bp^{2}\right),2^{2c}\left(a+bp^{4}\right)\right\} .
\]
\end{example}

Taking different numbers, one can deduce many different configurations.

\subsection{ A new multiplicative operation}

Let us take the set $\omega\times2\mathbb{N}$ and the binary operation
$\cdot$ on $\omega\times2\mathbb{N}$, where the operation $\cdot$
is defined as $\left(a,2b\right)\cdot\left(c,2d\right)=\left(a\cdot b,2^{2}c\cdot d\right)$.
Then, with this operation $\left(\omega\times2\mathbb{N},\cdot\right)$
forms a commutative semigroup. 

Now, We know that each $n\in\mathbb{N}$ can be written as $n=2^{x_{n}}\left(2y_{n}-1\right)$
in a unique way, where $x_{n}=f\left(n\right)$. So, the function
$\rho:\mathbb{N}\rightarrow\omega\times2\mathbb{N}$ defined by 
\[
\rho\left(n\right)=\left(x_{n},2y_{n}\right)=\left(f\left(n\right),\frac{n}{2^{f\left(n\right)}}+1\right)
\]
 is a bijection and it induces an associative operation $\otimes$
on $\mathbb{N}$, defined by, 
\[
p=m\otimes n
\]
\[
\text{if and only if},\ \varphi\left(p\right)=\varphi\left(m\right)\cdot\varphi\left(n\right)
\]

\[
\text{if and only if},\ \varphi\left(p\right)=\left(x_{m}\cdot x_{n},2y_{m}\cdot2y_{n}\right)
\]

\[
\text{if and only if},\ p=2^{x_{m}\cdot x_{n}}\left(2^{2}\left(y_{m}\cdot y_{n}\right)-1\right)
\]

\[
\text{if and only if},\ p=2^{f\left(m\right)\cdot f\left(n\right)}\left(\left(\frac{m}{2^{f\left(m\right)}}+1\right)\cdot\left(\frac{n}{2^{f\left(n\right)}}+1\right)-1\right)
\]

\[
\text{if and only if},\ p=2^{f\left(m\right)\cdot f\left(n\right)}\cdot\mathfrak{G}_{1,1}\left(\frac{m}{2^{f\left(m\right)}},\frac{n}{2^{f\left(n\right)}}\right).
\]
So, we have $m\otimes n=2^{f\left(m\right)\cdot f\left(n\right)}\cdot\mathfrak{G}_{1,1}\left(\frac{m}{2^{f\left(m\right)}},\frac{n}{2^{f\left(n\right)}}\right)$.

It can be easily verify that 
\[
a_{1}\otimes a_{2}\otimes\cdots\otimes a_{n}=2^{\prod_{i=1}^{n}f\left(a_{i}\right)}\cdot\mathfrak{G}_{1,1}\left(\frac{a_{1}}{2^{f\left(a_{1}\right)}},\frac{a_{2}}{2^{f\left(a_{2}\right)}},\ldots,\frac{a_{n}}{2^{f\left(a_{n}\right)}}\right).
\]

We will now use the Hales-Jewett theorem to deduce some new Ramsey
theoretic configurations.
\begin{thm}
Let $r\in\mathbb{N}$ and $a_{1},a_{2},\ldots,a_{n}$ are distinct
natural numbers. Then for every $r$-partition of $\mathbb{N}$, there
exist $x,y$ and $c$ in $\mathbb{N}$ such that 
\[
\left\{ x2^{f\left(a\right)^{c}}\mathfrak{G}_{1,1}\left(y,\left(\frac{a}{2^{f\left(a\right)}}\right)^{\left(c\right)}\right):a\in\left\{ a_{1},a_{2},\ldots,a_{n}\right\} \right\} 
\]
 is monochromatic.
\end{thm}

\begin{proof}
Let $\mathbb{A}=\left\{ a_{1},a_{2},\ldots,a_{n}\right\} $ and $r\in\mathbb{N}$.
Then choose the Hales-Jewett number $N=N\left(\mathbb{A},r\right)$.

Now consider the word space $\mathbb{A}^{N}$ and take the correspondence
map $g:\mathbb{A}^{N}\rightarrow\mathbb{N}$ defined by, 

$g\left(a_{1},a_{2},\ldots,a_{N}\right)$

$=a_{1}\otimes a_{2}\otimes\cdots\otimes a_{N}$

$=2^{\prod_{i=1}^{N}f\left(a_{i}\right)}\mathfrak{G}_{1,1}\left(\frac{a_{1}}{2^{f\left(a_{1}\right)}},\frac{a_{2}}{2^{f\left(a_{2}\right)}},\ldots,\frac{a_{N}}{2^{f\left(a_{N}\right)}}\right)$.

Now every $r$-partition on $\mathbb{N}$ induces a $r$-partition
on $\mathbb{A}^{N}$. Let $c$ be the number of variable position.
Let $d=N-c$, then consider 
\[
y=\mathfrak{G}_{1,1}\left(\frac{b_{1}}{2^{f\left(b_{1}\right)}},\frac{b_{2}}{2^{f\left(b_{2}\right)}},\ldots,\frac{b_{d}}{2^{f\left(b_{d}\right)}}\right)
\]
 and $x=\prod_{i=1}^{d}f\left(b_{i}\right),$ where $b_{i}'s$ are
the in non-variable positions. Then, from the Hales-Jewett theorem
and the above expression, 
\[
\left\{ x2^{f\left(a\right)^{c}}\cdot\mathfrak{G}_{1,1}\left(y,\left(\frac{a}{2^{f\left(a\right)}}\right)^{\left(c\right)}\right):a\in\mathbb{A}\right\} 
\]
 is monochromatic.
\end{proof}
\begin{cor}
Let $n,r\in\mathbb{N}$ and $p_{1},p_{2},\ldots,p_{n}$ be different
odd primes. Then for any $r$-partition of $\mathbb{N}$, there exist
$x,z,c\in\mathbb{N}$ such that the following configuration is monochromatic:
\[
\left\{ xz\cdot\left(p_{i}+1\right)^{c}-x:i\in\left\{ 1,2,\ldots,n\right\} \right\} 
\]
\end{cor}

\begin{proof}
Let $n,r\in\mathbb{N}$ and $p_{1},p_{2},\ldots,p_{n}$ be different
odd primes. Then for any $r$-partition of $\mathbb{N}$, from the
above theorem there exist $c,x,y\in$ $\mathbb{N}$ such that the
following configuration
\[
\left\{ x2^{f\left(p_{i}\right)^{c}}\cdot\mathfrak{G}_{1,1}\left(y,\left(\frac{p_{i}}{2^{f\left(p_{i}\right)}}\right)^{\left(c\right)}\right):i\in\left\{ 1,2,\ldots,n\right\} \right\} 
\]
 is monochromatic. Now, $f\left(p_{i}\right)=0$ and $\frac{p_{i}}{2^{f\left(p_{i}\right)}}=p_{i}$.
So, 

$x2^{f\left(p_{i}\right)^{c}}\cdot\mathfrak{G}_{1,1}\left(y,\left(\frac{p_{i}}{2^{f\left(p_{i}\right)}}\right)^{\left(c\right)}\right)$

$=x\cdot\mathfrak{G}_{1,1}\left(y,p_{i}^{\left(c\right)}\right)$

$=x\left\{ \left(y+1\right)\left(p_{i}+1\right)^{c}-1\right\} $.

Let $z=\left(y+1\right)$. Then we can say
\[
\left\{ xz\cdot\left(p_{i}+1\right)^{c}-x:i\in\left\{ 1,2,\ldots,n\right\} \right\} 
\]
 is monochromatic.
\end{proof}
\begin{cor}
Let $n,r\in\mathbb{N}$ and $p$ be an odd prime. Then for any $r$-partition
of $\mathbb{N}$, there exist $x,z,c\in\mathbb{N}$ such that the
following configuration is monochromatic:
\[
\left\{ xz\cdot\left(p^{i}+1\right)^{c}-x:i\in\left\{ 1,2,\ldots,n\right\} \right\} 
\]
\end{cor}

\begin{proof}
Consider the numbers $p,p^{2},\ldots,p^{n}$. Then $f\left(p^{i}\right)=0$
and $\frac{p^{i}}{2^{f\left(p^{i}\right)}}=p^{i}$ for all $i\in\left\{ 1,2,\ldots,n\right\} $.
Then for any $r$-partition of $\mathbb{N}$, from the above theorem
there exist $x,y\in$ $\mathbb{N}$ such that the following monochromatic
configuration: 
\[
\left\{ x2^{f\left(p^{i}\right)^{c}}\cdot\mathfrak{G}_{1,1}\left(y,\left(\frac{p^{i}}{2^{f\left(p^{i}\right)}}\right)^{\left(c\right)}\right):i\in\left\{ 1,2,\ldots,n\right\} \right\} .
\]
Let $z=y+1$. Then, 

$x2^{f\left(p^{i}\right)^{c}}\cdot\mathfrak{G}_{1,1}\left(y,\left(\frac{p^{i}}{2^{f\left(p^{i}\right)}}\right)^{\left(c\right)}\right)$

$=x\cdot\mathfrak{G}_{1,1}\left(y,p^{i\,\left(c\right)}\right)$

$=x\left\{ \left(y+1\right)\left(p^{i}+1\right)^{c}-1\right\} $

$=x\left\{ z\left(p^{i}1\right)^{c}-1\right\} $

$=xz\left(p^{i}+1\right)^{c}-x.$

So, we have the desired monochromatic configuration.
\end{proof}
The following two examples shows the type of monochromatic configurations
we get from the above corollary.
\begin{example}
Let us consider the numbers $2^{2}p,2p^{4}$ for some odd prime $p$.
Now $f\left(2^{2}p\right)=2$, $f\left(2p^{4}\right)=1$, and $\frac{2^{2}p}{2^{f\left(2^{2}p\right)}}=p$,
$\frac{2p^{4}}{2^{f\left(2p^{4}\right)}}=p^{4}$. So for any $r$-partition
of $\mathbb{N}$, there exist $x,y,c\in\mathbb{N}$ such that the
following two patterns 
\[
x2^{2^{c}}\cdot\mathfrak{G}_{1,1}\left(y,p^{(c)}\right)=2^{2^{c}}x\left\{ \left(y+1\right)\left(p+1\right)^{c}-1\right\} 
\]
 and
\[
2x\cdot\mathfrak{G}_{1,1}\left(y,\left(p^{4}\right)^{(c)}\right)=2x\left\{ \left(y+1\right)\left(p^{4}+1\right)^{c}-1\right\} ,
\]
are monochromatic.
\end{example}

\begin{example}
Let us consider the numbers $pq,2p,q^{2}$ for some odd primes $p$
and $q$. Now, $f\left(pq\right)=0$, $f\left(2p\right)=1$ and $f\left(q^{2}\right)=0$.
And $\frac{pq}{2^{f\left(pq\right)}}=pq$, $\frac{2p}{2^{f\left(2p\right)}}=p$
and $\frac{q^{2}}{2^{f\left(q^{2}\right)}}=q^{2}$. So, for any $r$-partition
of $\mathbb{N}$, there exist $x,y,c\in\mathbb{N}$ such that the
following patterns 
\[
x\cdot\mathfrak{G}_{1,1}\left(y,\left(pq\right)^{(c)}\right)=x\left\{ \left(y+1\right)\left(pq+1\right)^{c}-1\right\} ,
\]
\[
2x\cdot\mathfrak{G}_{1,1}\left(y,p^{(c)}\right)=2x\left\{ \left(y+1\right)\left(p+1\right)^{c}-1\right\} 
\]
 and 
\[
x\cdot\mathfrak{G}_{1,1}\left(y,\left(q^{2}\right)^{(c)}\right)=x\left\{ \left(y+1\right)\left(q^{2}+1\right)^{c}-1\right\} ,
\]
 are monochromatic.
\end{example}
\vspace{.2in}
\noindent \textbf{\large{Acknowledgement:}} The authors are thankful to Prof. Dibyendu De for his guidance and continuous inspiration. We also acknowledge his helpful comments on the previous draft of the paper.

\Addresses 

\end{document}